\documentclass{article}
\pdfoutput=1
\usepackage[utf8]{inputenc}
\usepackage{amsmath}
\usepackage{amssymb}
\usepackage{multirow, hhline}
\usepackage[table]{xcolor}
\usepackage[para,online,flushleft]{threeparttable}
\usepackage{microtype}   
\usepackage[letterpaper,margin=1.0in]{geometry}
\usepackage{graphicx}

\usepackage{hyperref}       
\usepackage{url}            
\usepackage{booktabs}       
\usepackage{amsfonts}       
\usepackage{bbding}
\usepackage{xcolor}         
\usepackage[numbers,sort]{natbib}
\usepackage{color}
\usepackage{subcaption}
\usepackage{graphicx}
\usepackage{tablefootnote}
\usepackage{algorithm}
\usepackage{algpseudocode}
\usepackage{bbding}
\usepackage{footnote}
\makesavenoteenv{tabular}
\usepackage{threeparttable}
\usepackage{enumitem}

\usepackage{amsmath}\allowdisplaybreaks
\usepackage{amsfonts,bm}
\usepackage{amssymb}

\def\eqref#1{equation~(\ref{#1})}

\def\1{\bf{1}}

\newcommand{\Norm}[1]{\left\| #1 \right\|}

\def\inner#1#2{\langle #1, #2 \rangle}
\newcommand{\norm}[1]{\| #1\|}

\def\vDelta{{\bm{\Delta}}}
\def\vxi{{\bm{\xi}}}
\def\vzero{{\bf{0}}}

\def\va{{\bf{a}}}

\def\vg{{\bf{g}}}

\def\vu{{\bf{u}}}
\def\vv{{\bf{v}}}

\def\vx{{\bf{x}}}
\def\vy{{\bf{y}}}
\def\vz{{\bf{z}}}

\def\fA{{\mathcal{A}}}

\def\fD{{\mathcal{D}}}

\def\fF{{\mathcal{F}}}

\def\fO{{\mathcal{O}}}

\def\fX{{\mathcal{X}}}

\def\sB{{\mathbb{B}}}

\def\BE{{\mathbb{E}}}

\def\BI{{\mathbb{I}}}

\def\BN{{\mathbb{N}}}

\def\BP{{\mathbb{P}}}

\def\BR{{\mathbb{R}}}
\def\BS{{\mathbb{S}}}

\def\BU{{\mathbb{U}}}

\DeclareMathOperator*{\argmax}{arg\,max}
\DeclareMathOperator*{\argmin}{arg\,min}

\usepackage{amsthm}
\usepackage{etoolbox}
\theoremstyle{plain}
\newtheorem{theorem}{Theorem}
\AfterEndEnvironment{thm}{\noindent\ignorespaces}
\newtheorem{definition}{Definition}
\AfterEndEnvironment{defn}{\noindent\ignorespaces}

\AfterEndEnvironment{exmp}{\noindent\ignorespaces}
\newtheorem{lemma}{Lemma}
\AfterEndEnvironment{lem}{\noindent\ignorespaces}
\newtheorem{assumption}{Assumption}
\AfterEndEnvironment{asm}{\noindent\ignorespaces}
\newtheorem{remark}{Remark}
\AfterEndEnvironment{remark}{\noindent\ignorespaces}

\AfterEndEnvironment{cor}{\noindent\ignorespaces}
\newtheorem{proposition}{Proposition}
\AfterEndEnvironment{prop}{\noindent\ignorespaces}

\makeatletter
\def\Ddots{\mathinner{\mkern1mu\raise\p@
\vbox{\kern7\p@\hbox{.}}\mkern2mu
\raise4\p@\hbox{.}\mkern2mu\raise7\p@\hbox{.}\mkern1mu}}
\makeatother

\makeatletter
\newcommand*{\rom}[1]{\expandafter\@slowromancap\romannumeral #1@}
\makeatother

\usepackage{tcolorbox}
\usepackage{pifont}
\definecolor{mydarkgreen}{RGB}{39,130,67}
\definecolor{mydarkred}{RGB}{192,25,25}
\definecolor{bgcolor}{rgb}{0.93,0.99,1}
\definecolor{bgcolor2}{rgb}{0.8,1,0.8}
\definecolor{bgcolor3}{rgb}{0.50,0.90,0.50}

\usepackage{relsize} 

\begin{document}

\title{A Parameter-Free and Near-Optimal Zeroth-Order Algorithm for \\ Stochastic Convex Optimization}

\author{
    Kunjie Ren\thanks{School of Data Science, Fudan University; kjren24@m.fudan.edu.cn} \qquad \quad
     Luo Luo\thanks{School of Data Science, Fudan University; luoluo@fudan.edu.cn}
 }
 \date{}
 

\maketitle

\begin{abstract}
This paper considers zeroth-order optimization for stochastic convex minimization problem.
We propose a \underline{p}arameter-free st\underline{o}chastic z\underline{e}roth-order \underline{m}ethod (POEM) by introducing a stepsize scheme based on the distance over finite difference and an adaptive smoothing parameter.
We provide the theoretical analysis to show that POEM achieves the  near-optimal stochastic zeroth-order oracle complexity. 
We further conduct the numerical experiments to demonstrate POEM outperforms existing zeroth-order methods in practice.
\end{abstract}

\section{Introduction}
\vskip 0.2cm
This paper studies the stochastic optimization problem 
\begin{align}\label{main:prob}
    \min_{\vx\in \fX} f(\vx) \triangleq \BE_{\vxi\sim \Xi}[F(\vx; \vxi)]
\end{align}
where the domain $\fX\subseteq \BR^d$ is a compact convex set, the random variable $\vxi$ follows the distribution $\Xi$, and the stochastic component function $F(\vx;\vxi)$ is convex and Lipschitz continuous in $\vx$ on $\fX$ for given $\vxi$.
We focus on stochastic zeroth-order optimization for solving Problem~(\ref{main:prob}), 
which allows the algorithm to query stochastic 
function values at given point.
This problem is relevant to the scenario that accessing the (stochastic) first-order information is expensive or infeasible, 
including the applications in bandit optimization \cite{flaxman2004online,agarwal2010optimal,duchi2015optimal, shamir2017optimal}, 
adversarial training \cite{goodfellow2014explaining,shaham2018understanding}, 
reinforcement learning \cite{balasubramanian2018zeroth,mania2018simple}, 
and other black-box models \cite{liu2016delving,ilyas2018prior,liu2020primer}.

The finite difference methods are widely used in zeroth-order optimization, which estimates the first-order information of the objective function by random directions
\cite{kiefer1952stochastic,ghadimi2013stochastic,duchi2015optimal,nesterov2017random,nazari2020adaptive,gasnikov2022power,lin2022gradient,rando2024optimal,chen2023faster,kornowski2024algorithm}.
For stochastic convex problem with Lipschitz continuous components,
\citet{nesterov2017random} developed random search methods with sublinear convergence rates.
Later, \citet{duchi2015optimal} proposed a stochastic algorithm by using two random sequences to construct the stochastic finite differences, achieving the sharper dependence on the dimension than the result of \citet{nesterov2017random}.
They further provide the lower bound to show that their algorithm achieves the near-optimal stochastic zeroth-order oracle (SZO) complexity.
Consequently, \citet{shamir2017optimal} proposed an algorithm with single random sequence that is established by the uniform distribution on the unit ball, which is optimal and easy to implement.
Recently, \citet{rando2024optimal} studied the finite difference with random orthogonal directions, also achieving the optimal ZSO complexity.
In addition, \citet{ghadimi2013stochastic,nesterov2017random,wibisono2012finite,balasubramanian2018zeroth} considered the zeroth-order optimization for the smooth problem and
\citet{lin2022gradient,chen2023faster,kornowski2024algorithm} studied the nonsmooth nonconvex problem.
There are several limitations in existing zeroth-order optimization methods \cite{kiefer1952stochastic,ghadimi2013stochastic,duchi2015optimal,nesterov2017random,nazari2020adaptive,gasnikov2022power,lin2022gradient,rando2024optimal}. 
Specifically, their performance is heavily affected by the parameter settings and the optimal convergence rate typically requires the appropriate step size, which depends on the knowledge of the problem properties (e.g., the Lipschitz constant) and the iteration budget.
In addition, the smoothing parameter in the finite difference usually depends on target accuracy or decays rapidly, which may arise the issue on the numerical stability in practice.

We desire to establish the adaptive stochastic optimization methods to avoid the pain of tuning parameters.
Most of existing works focus on first-order methods. 
For example, \citet{rolinek2018l4, vaswani2019painless, paquette2020stochastic} introduced line search methods for stochastic optimization. 
\citet{tan2016barzilai, berrada2020training, loizou2021stochastic,wang2023generalized} extended the Barzilai--Borwein (BB) step size \cite{barzilai1988two} and the Polyak step size \cite{polyak1987introduction} to stochastic settings.
For training deep neural networks, the adaptive algorithms such as AdaGrad \cite{duchi2011adaptive}, Adam \cite{kingma2014adam}, and their variants \cite{tieleman2012lecture, zeiler2012adadelta, shazeer2018adafactor,wang2024provable,zhang2024adam} exploit the specific problem structure and achieve success in many applications. 
However, these methods still require the appropriate initialization, which may significantly affect convergence rates in both theoretical and empirical. 

Ideally, we shall design the parameter-free optimization method which has the near-optimal convergence rate and almost does not require the knowledge of problem properties \cite{streeter2012no,defazio2023learning,lan2023optimal,li2023simple}.
There are several such kind of methods have been established by online learning techniques for stochastic convex optimization 
\cite{luo2015achieving, orabona2016coin, cutkosky2018black, bhaskara2020online, mhammedi2020lipschitz, jacobsen2022parameter}, while their implementations are quite complicated.
In practice, \citet{orabona2017training, chen2022better} adopt the classical stochastic gradient descent (SGD) framework with coin-betting techniques, which perform well in training neural network.
Later, \citet{carmon2022making} showed that using the bisection step to determine the step size in SGD leads to a parameter-free algorithm with near-optimal convergence rates.
Furthermore, \citet{ivgi2023dog} developed a parameter-free step size schedule called Distance over Gradients (DoG), which uses the distance from the initial point and the norm of stochastic gradients \cite{you2017large, shazeer2018adafactor, bernstein2020distance} to adjust the step sizes. 
DoG achieves near-optimal convergence rates and has good performance in practice. 
However, all existing parameter-free methods are designed for first-order optimization.
The parameter-free zeroth-order optimization includes additional challenges such as making both the step size and the smoothing parameter be tuning-free and the addressing the dependence on the dimension in the convergence rates.

In this paper, we propose a \underline{p}arameter-free st\underline{o}chastic z\underline{e}roth-order \underline{m}ethod (POEM) by using a stepsize schedule based on the distance over finite difference and the adaptive smoothing parameter.
For the stochastic convex optimization (\ref{main:prob}), we show that the initialization only affects the convergence rates of POEM by a logarithmic factor. 
We provide the high probability convergence guarantees for POEM to show that it achieves the near-optimal SZO complexity. 
We compare the results of POEM with related work in Table~\ref{table:szo}.
We also study the problem with unbounded domain by showing the impossibility of ideal parameter-free algorithm in such setting.
Moreover, we conduct the numerical experiments to  demonstrate the advantage of our method in practice.

\begin{table*}[t]
\centering
\caption{We present the ZSO complexity, the step size $\eta_t$, and the smoothing parameter $\mu_t$ for finding the $\epsilon$-suboptimal solution of Problem~(\ref{main:prob}), where $t$ is the iteration counter, $T$ is the total iteration budget, ${\bar r}_t$ and $G_t$ are defined in Algorithm \ref{alg:poem}.}\label{table:szo}
\vskip0.05cm 
\begin{tabular}{cccccc}
\hline
Algorithms & ~Parameter-Free~ & ~~SZO Complexity~~& ~~$\eta_t$~~ & ~~$\mu_t$~~  \\
\hline\hline\addlinespace
\begin{tabular}{c}
RSNSO$^\sharp$   \\ \cite{nesterov2017random}
\end{tabular}
 & No & $\fO\left(\dfrac{d^2L^2 s_0^2}{\epsilon^2}\right)$ & $\dfrac{s_0}{dL\sqrt{T}}$ & $s_0\sqrt{\dfrac{d}{T}}$  \\ \addlinespace
\begin{tabular}{c}
TPGE$^\ddag$ \\ \cite{duchi2015optimal}
\end{tabular}
 & No &    $\tilde\fO\left(\dfrac{dL^2 D_\fX^2}{\epsilon^2}\right)$ & ~~$\dfrac{D_\fX}{L\sqrt{d\log(2d)t}}$~~ & ~~~$\dfrac{D_\fX}{t}$~~and~~$\dfrac{D_\fX}{d^2t^2}$\\ \addlinespace
\begin{tabular}{c}
   TPBCO \\ \cite{shamir2017optimal}
\end{tabular} & No &  $\fO\left(\dfrac{dL^2 D_\fX^2}{\epsilon^2} \right)$ & $\dfrac{D_\fX}{L\sqrt{dT}}$ & $D_\fX\sqrt{\dfrac{d}{T}}$\\\addlinespace
\begin{tabular}{c}
POEM \\ (Theorem \ref{thm})
\end{tabular} & \textbf{Yes} & 
 $\tilde\fO\left(\dfrac{dL^2 D_\fX^2}{\epsilon^2}\right)$ & $\dfrac{\bar{r}_{t}}{\sqrt{G_{t}}}$ & $\bar{r}_t\sqrt{\dfrac{d}{t+1}}$ &  \\ 
\addlinespace
\hline\addlinespace
\begin{tabular}{c}
  Lower bound \\ \cite{duchi2015optimal}
\end{tabular}
  & -- & 
  $\Omega\left(\dfrac{dL^2 D_\fX^2}{\epsilon^2}\right)$ & -- & -- \\ \addlinespace \hline
\end{tabular}  
{\begin{itemize}[topsep=3.5pt,itemsep=1.8pt,partopsep=3.5pt, parsep=2pt,leftmargin=0.75cm]
    \item[\scriptsize$\sharp$] \scriptsize \!The RSNSO \cite{nesterov2017random} does not require assuming the domain is bounded since its complexity depends on the distance from the initial point $\vx_0$  to the solution $\vx_\star$, i.e., $s_0=\norm{\vx_0-\vx_\star}$, rather than the diameter $D_\fX$. We detailed discuss the case without bounded domain assumption in Section \ref{sec:unbounded}. \\[-1.5cm]
    \item[\scriptsize$\ddag$] \scriptsize \!The TPGE \cite{duchi2015optimal} has two two sequence of stochastic finite differences, corresponding to two different smoothing parameters settings.
\end{itemize}}\vskip-0.2cm
\end{table*}

\section{Preliminaries}\label{sec:pre}
In this section, we formalize the problem setting  and introduce the smoothing technique in zeroth-order optimization.

\subsection{Notation and Assumptions}
Throughout the paper, we use $\norm{\cdot}$ to denote the Euclidean norm.
We denote the unit ball as \mbox{$\sB^d\triangleq\{\vu\in \BR^d: \norm{\vu}\leq 1\}$}
and the unit sphere as $\BS^{d-1}\triangleq \{\vv\in \BR^d: \norm{\vv}= 1\}$.
We let $\BU(\sB^d)$ and $\BU(\BS^{d-1})$ be the uniform distribution over the unit ball and unit sphere, respectively.
Moreover, we use $\tilde\fO(\cdot)$ to hide the logarithmic factor.

We impose the following assumptions for Problem  (\ref{main:prob}).

\begin{assumption}\label{asm:domain}
The domain $\fX\subseteq\BR^d$ is compact and convex.
Furthermore, we denote the diameter of $\fX$ as
\begin{align*}
    D_\fX \triangleq \max_{\vx,\vy\in \fX} \norm{\vx-\vy}<\infty.
\end{align*}
\end{assumption}

According to Assumption \ref{asm:domain}, the objective function attains its minimum on the compact set $\fX$.
Thus, we define the optimal solution of Problem (\ref{main:prob}) as follows.

\begin{definition}\label{def:x_star}
We let $\vx_\star\in \fX$ be an optimal solution of Problem (\ref{main:prob}) such that $f(\vx_\star)=\min_{\vx\in\fX} f(\vx)$.  
\end{definition}

We define the Euclidean projection onto the
domain $\fX$ as follows.

\begin{definition}\label{def:project}
We define the Euclidean projection onto the compact and convex set $\fX\subseteq\BR^d$ at the point $\vx\in\BR^d$ as
\begin{align*}
    \Pi_{\fX}(\vx) \triangleq \argmin_{\vy\in\fX} \Norm{\vx-\vy}.
\end{align*}
\end{definition}

We desire the iterative algorithm find the approximate solution of problem (\ref{main:prob}) which is defined as follows.
  
\begin{definition} \label{def:sol}
We say $\hat \vx$ is an $\epsilon$-suboptimal solution of \mbox{Problem (\ref{main:prob})} if holds $f(\hat \vx) - f(\vx_\star) \le \epsilon$ for given $\epsilon>0$.
\end{definition}

We also suppose the stochastic component $F(\vx;\vxi)$ is convex and Lipschitz continuous in $\vx$.

\begin{assumption} \label{asm:convex}
The stochastic component $F(\vx;\vxi)$ is \mbox{convex} in $\vx$ for given $\vxi$.
\end{assumption}

\begin{assumption} \label{asm:lipschitz}
There exists a constant $L \geq  0$ such that we have  $\norm{F(\vx;\vxi)-F(\vy;\vxi)}\leq L\norm{\vx-\vy}$ for all $\vx, \vy$, and~$\vxi$.
\end{assumption}

We assume the algorithm for solving Problem (\ref{main:prob}) can access the stochastic zeroth-order oracle that returns the unbiased stochastic function value evaluations at two points.

\begin{assumption}\label{asm:oracle}
The stochastic zeroth-order oracle can outputs the stochastic evaluations $F(\vx;\vxi)$ and $F(\vy;\vxi)$ for given $\vx\in\BR^d$ and $\vy\in\BR^d$ such that $\BE_\vxi[F(\vx;\vxi)]=f(\vx)$ and $\BE_\vxi[F(\vy;\vxi)]=f(\vy)$, where $\vxi$ is drawn from $\Xi$.
\end{assumption}

\subsection{Randomized Smoothing}

Randomized smoothing is a popular technique in zeroth-order optimization, which established smooth surrogate of the objective functions by the perturbation along with the random directions \cite{duchi2012randomized,gasnikov2022power,shamir2017optimal,yousefian2012stochastic, nesterov2017random,lin2022gradient}.
In this work, we focus on random smoothing based on the uniform distribution on the unit ball \cite{duchi2012randomized,gasnikov2022power,shamir2017optimal}.
That is, we define the smooth surrogate of the objective function $f(\vx)$ as
\begin{align*}
f_\mu(\vx) \triangleq \BE_{\vu\sim \BU(\sB^d)}[f(\vx+\mu\vu)],
\end{align*}
where $\mu>0$ is the smoothing parameter.
The following lemma shows that the surrogate $f_\mu(\vx)$ preserves the convexity, and the difference between $f(\vx)$ and $f_\mu(\vx)$ can be bounded by the smoothing parameter $\mu$ \cite{shamir2017optimal}. 
\begin{lemma}[{\citet[Lemma 8]{shamir2017optimal}}]\label{lem:error}
Under Assumptions \ref{asm:convex} and \ref{asm:lipschitz},
the smooth surrogate $f_\mu(\vx)$ is convex and satisfies
$|f_\mu(\vx)-f(\vx)|\leq L\mu$ for all $\vx\in\BR^d$.
\end{lemma}

The next lemma shows that the surrogate function $f_\mu(\vx)$ is differentiable regardless of whether $f_\mu(\vx)$ is differentiable, and its stochastic gradient can be presented by the form of stochastic finite difference.

\begin{lemma}[{\citet[Lemma 3.4]{flaxman2004online}}]\label{lem:grad}
For the continuous function $f:\BR^d\to\BR$, the gradient of its surrogate $f_\mu$ is given by
\begin{align*}
\nabla f_\mu(\vx) = \BE_{\vv\sim \BU(\BS^{d-1})}\bigg[\frac{d}{2\mu}(f(\vx+\mu\vv)-f(\vx-\mu\vv))\vv\bigg],
\end{align*}
where $\vx\in\BR^d$ and $\mu>0$.
\end{lemma}

Based on Lemma \ref{lem:grad}, we define the stochastic finite difference as follows
\begin{align}\label{eq:sfd}
\begin{split}    
\!\!\vg(\vx, \mu; \vv,\vxi) \triangleq \frac{d}{2\mu}(F(\vx+\mu\vv;\vxi)-F(\vx-\mu\vv;\vxi))\vv,\!
\end{split}
\end{align}
where $\vx\in \fX$, $\mu>0$, $\vv\sim \BU(\BS^{d-1})$ and $\vxi\sim \Xi$. 
Under Assumption \ref{asm:oracle}, the stochastic zeroth-order oracle $F(\vx;\vxi)$ is an unbiased estimator of $f(\vx)$, which implies the stochastic finite difference $\vg(\vx, \mu; \vv,\vxi)$ is also an unbiased estimator of $\nabla f_\mu(\vx)$, i.e.,
\begin{align*}
\BE_{\vv\sim \BU(\BS^{d-1}),\,\vxi\sim\Xi}[\vg(\vx, \mu; \vv,\vxi) ] = \nabla f_\mu(\vx).
\end{align*}

\section{Parameter-Free Stochastic Zeroth-Order Optimization}\label{sec:POEM}
We propose our \underline{p}arameter-free st\underline{o}chastic z\underline{e}roth-order \underline{m}ethod (POEM) in Algorithm \ref{alg:poem}.
We established POEM by the framework of projected SGD iteration
\begin{align}\label{eq:iteration}
    \vx_{t+1} = \Pi_{\fX}(\vx_t-\eta_t\vg_t),
\end{align}
where $\eta_t>0$ is the step size, 
and the finite difference
\begin{align}\label{eq:gt-iter}
    \vg_t\triangleq \vg(\vx_t,\mu_t;\vv_t,\vxi_t)     
\end{align}
follows equation (\ref{eq:sfd}) with the smoothing parameter $\mu_t>0$ and random variables $\vv_t\sim \BU(\BS^{d-1})$ and $\vxi_t\sim \Xi$.

We target to make both the step size $\eta_t$ and the smoothing parameter $\mu_t$ in equations (\ref{eq:iteration}) and (\ref{eq:gt-iter}) to be tuning-free and achieve the near-optimal convergence rate, which is more challenging than the stochastic parameter-free first-order methods that mainly focus on the step size.
Inspired by the strategy of DoG \cite{ivgi2023dog}, we schedule the step size $\eta_t$ based on the distance to initial point over the norm of stochastic finite difference.
Specifically, we define the sum of squared norms as
$G_{t}\triangleq\sum_{k=0}^{t}\norm{\vg_k}^2$, the distance to the initial point as $r_t\triangleq\norm{\vx_t-\vx_0}$,
and the maximum distance as $\bar{r}_{t}\triangleq\max _{k \leq t} r_{k} \vee r_{\epsilon}$, where $r_{\epsilon}>0$ is the initial movement.

We then set the step size at the $t$-th iteration as
\begin{align}\label{eq:eta-r-g}
     \eta_{t}\triangleq\frac{\bar{r}_{t}}{\sqrt{G_{t}}}.
\end{align}
In the initialization, we require the movement $r_\epsilon\in (0,D_\fX]$.
In Section \ref{sec:complexity}, we will show that the setting of $r_\epsilon$ only affects the logarithmic term in the convergence rate and it does not heavily affect the performance of the algorithm in practice.

\begin{algorithm}[t]
\caption{POEM} \label{alg:poem}
\begin{algorithmic}[1]
\State \textbf{Input:} $\vx_0\in \fX$,\quad $r_\epsilon\in(0, D_\fX]$,\quad $T\geq 1$  \vskip0.1cm
\State $\bar{r}_{-1}=r_\epsilon$,~~ $G_{-1}=0$ \vskip0.12cm
\State \textbf{for} $t = 0, \dots, T-1$ \textbf{do} \vskip0.12cm
        \State \quad $\bar{r}_t = \max\{\bar{r}_{t-1}, \norm{\vx_t-\vx_0}\}$ \vskip0.15cm   
       \State \quad $\mu_t = \bar{r}_t\sqrt{\dfrac{d}{t+1}}$ \vskip0.12cm
        \State \quad $\vv_t \sim \BU(\BS^{d-1})$,~~
         $\vxi_t \sim \Xi$  \vskip0.15cm
        \State \quad $\vg_t = \dfrac{d}{2\mu_t}(F(\vx_t+\mu_t\vv_t;\vxi_t)-F(\vx_t-\mu_t\vv_t;\vxi_t))\vv_t$ \vskip0.15cm
        \State \quad $G_t = G_{t-1}+\norm{\vg_t}^2$ \vskip0.15cm
        \State \quad $\eta_t =\dfrac{\bar{r}_t}{\sqrt{G_t}}$  \vskip0.15cm
        \State \quad $\vx_{t+1} = \Pi_{\fX}(\vx_t-\eta_t\vg_t)$  \vskip0.12cm
	\State\textbf{end for} 
\State \textbf{Output:} $\bar{\vx}_{\tau_T}$ where $\displaystyle{\tau_T\triangleq \arg\max_{t\leq T}\sum_{k=0}^{t-1}\frac{\bar{r}_k}{\bar{r}_t}}$    \vskip0.12cm
\end{algorithmic}
\end{algorithm}

Moreover, we set the smoothing parameter at the $t$-th as
\begin{align}\label{def:smooth}
    \mu_t \triangleq \bar{r}_t\sqrt{\frac{d}{t+1}},
\end{align}
which is adaptive and larger than the smoothing parameters required in the existing stochastic zeroth-order methods \cite{nesterov2017random,shamir2017optimal,duchi2015optimal,ghadimi2013stochastic,wibisono2012finite,rando2024optimal},\
e.g., \citet{nesterov2017random,shamir2017optimal}
takes $\mu_t=\fO(\sqrt{d/T}\,)$ in their analysis, which depends on the total iteration budget $T$;
\citet{duchi2015optimal} takes $\mu_t=\fO(1/(dt)^2)$, which may be quite small for high dimensional problems.
Recall the smoothing parameter $\mu_t$ appears in the denominator of the stochastic finite difference (\ref{eq:sfd}). Therefore, the larger $\mu_t$ makes the iteration more stable.

\section{The Complexity Analysis}\label{sec:complexity}
We show that POEM (Algorithm \ref{alg:poem}) achieves the near-optimal SZO complexity. 
The detailed proofs for results in this section are deferred to Appendix \ref{apx:complexity}.

Our analysis considers the weighted average of the sequence generated from the iterations of POEM, that is
\begin{align}\label{eq:seq}
    \bar{\vx}_t \triangleq \frac{1}{\sum_{k=0}^{t-1} \bar{r}_{k}} \sum_{k=0}^{t-1} \bar{r}_{k} \vx_k.
\end{align}
Since the objective $f$ is convex (Assumption \ref{asm:convex}), we bound the function value gap by Jensen's inequality as
\begin{align}\label{eq:bound-Jensen}
f(\bar{\vx}_t)-f(\vx_\star) \leq \frac{1}{\sum_{k=0}^{t-1} \bar{r}_{k}}\sum_{k=0}^{t-1} \bar{r}_{k} (f(\vx_k)-f(\vx_\star)).
\end{align}
We combine Lemma \ref{lem:error} and inequality (\ref{eq:bound-Jensen}) to achieve
\begin{align}\label{eq:gapsmooth}
\begin{split}    
     f(\bar{\vx}_t)-f(\vx_\star) 
\leq & \frac{1}{\sum_{k=0}^{t-1} \bar{r}_{k}}\sum_{k=0}^{t-1} \bar{r}_{k} (f_{\mu_k}(\vx_k)-f_{\mu_k}(\vx_\star)+2L\mu_k) \\
\leq & \frac{1}{\sum_{k=0}^{t-1} \bar{r}_{k}} \sum_{k=0}^{t-1} \bar{r}_{k}( \inner{\nabla f_{\mu_k}(\vx_k)}{\vx_k-\vx_\star}+2L\mu_k),
\end{split}
\end{align}
where the second inequality uses the convexity of $f_{\mu_k}$.
We split the sum in the last line of equation (\ref{eq:gapsmooth}) into
\begin{align}\label{eq:threecomponentssmooth}
\begin{split}    
    \underbrace{\sum_{k=0}^{t-1} \bar{r}_k\inner{\vg_k}{\vx_k-\vx_\star}}_{\text{the weighted regret}} + \underbrace{\sum_{k=0}^{t-1} \bar{r}_k\inner{\vDelta_k}{\vx_k-\vx_\star}}_{\text{the noise from }\vg_k} +  \underbrace{\sum_{k=0}^{t-1}2L\bar{r}_k\mu_k}_{\text{the noise from }\mu_k},
\end{split}
\end{align}
where $\vDelta_{k}\triangleq \nabla f_{\mu_k}(\vx_k)-\vg_k$. 
The regret terms are typically included in the complexity analysis of stochastic zeroth-order and first-order  algorithms \cite{shalev2012online, duchi2015optimal, balasubramanian2018zeroth,ivgi2023dog}.
For the proposed POEM, we should address the weighted regret terms scaled by $\{\bar{r}_k\}_{k=0}^{t-1}$.
The noise from $\vg_k$ mainly depends on the difference between the gradient of $f_{\mu_k}$ and its unbiased gradient estimator $\vg_k$.
The noise from~$\mu_k$ is associated with the difference between the objective $f$ and its smooth surrogate  $f_{\mu_k}$, which is not contained in the analysis of stochastic first-order methods.

We now present the upper bounds for the terms of the weighted regret, the noise from $\vg_k$, and the noise from~$\mu_k$ in equation (\ref{eq:gapsmooth}), respectively.
We define 
\begin{align*}
s_t\triangleq \norm{\vx_t-\vx_\star} 
\qquad\text{and}\qquad \bar{s}_t\triangleq \max_{k\leq t} s_k.
\end{align*}
The following lemma \cite{ivgi2023dog} provides the upper bound for the weighted regret for SGD-type iterations, which does not depend on the specific settings of the gradient estimator and the step size and  naturally holds for our POEM (Algorithm \ref{alg:poem}).

\begin{lemma}[{\citet[Lemma 3.4]{ivgi2023dog}}]\label{lem:weighted-regret}
The weighted regret for the iteration scheme (\ref{eq:iteration}) holds that
\begin{align*}
     \sum_{k=0}^{t-1} \bar{r}_k\inner{\vg_k}{\vx_k-\vx_\star} \leq \bar{r}_{t}\left(2 \bar{s}_{t}+\bar{r}_{t}\right) \sqrt{G_{t-1}}.
\end{align*}
\end{lemma}

For the noise from $\vg_k$, we first establish the upper bounds of $\norm{\vg_t}$ and $\BE[\norm{\vg_t}^2]$.

\begin{lemma}[{\citet[Lemma 10]{shamir2017optimal}}]\label{lem:upper-bound-for-g}
Under Assumption \ref{asm:lipschitz}, POEM (Algorithm \ref{alg:poem}) holds that
\begin{align*}
    \norm{\vg_t}\leq Ld \quad\text{and}\quad \BE[\norm{\vg_t}^2]\leq cL^2d,
\end{align*}
where $c>0$ is a numerical constant.
\end{lemma}

\begin{remark}
This lemma provides the upper bounds that depend on $\fO(d)$ for both $\norm{\vg_t}$ and $\BE[\norm{\vg_t}^2]$, which is important to achieve the optimal dependence on $d$ in the final convergence rates.
Compared with the proof of \citet[Lemma 10]{shamir2017optimal}, this paper only considers the Euclidean norm and we provide a more concise analsis without using the fourth-order moment of the gradient estimator (see Appendix \ref{apx:upper-bound-for-g}).
\end{remark}

According to Lemma \ref{lem:upper-bound-for-g},
we can provide the upper bound for $\Norm{\vDelta_k}=\Norm{\nabla f_{\mu_k}(\vx_k)-\vg_k}$.
Consequently, we apply the concentration inequality for martingale differences \cite{howard2021time, ivgi2023dogsgdsbestfriend} to bound the noise from $\vg_k$.

\begin{lemma}\label{lem:noise-from-g}
Under Assumptions \ref{asm:lipschitz} and \ref{asm:oracle}, for all \mbox{$\delta\in (0,1)$}, POEM (Algorithm \ref{alg:poem}) holds that
\begin{align*}
    \BP\bigg(\exists t\leq T: \bigg|\sum_{k=0}^{t-1}\bar{r}_k\inner{\vDelta_k}{\vx_k-\vx_\star}\bigg|\geq b_t\bigg)\leq \delta,
\end{align*}
where $b_t=8 \bar{r}_{t-1} \bar{s}_{t-1} \sqrt{\theta_{t, \delta} G_{t-1}+4L^2d^2\theta_{t, \delta}^{2}}\,$, 
and $\theta_{t, \delta}\triangleq\log (60 \log(6 t/\delta))$.
\end{lemma}

The following lemma upper bounds the noise from $\mu_k$.
\begin{lemma}\label{lem:noise-from-mu}
The settings of POEM (Algorithm \ref{alg:poem}) holds 
\begin{align*}
\sum_{k=0}^{t-1} 2L\bar{r}_k\mu_k \leq 4L\bar{r}_{t-1}^2\sqrt{dt},
\end{align*}
where $\log_+(\cdot)\triangleq \log(\cdot)+1$.
\end{lemma}

\begin{remark}
The settings of smoothing parameter (\ref{def:smooth}) implies  that $\mu_k=\fO(\sqrt{d/k}\,)$, then we can use an integral to establish an upper bound for the series to prove Lemma \ref{lem:noise-from-mu}.
We present the detailed proof in Appendix \ref{apx:noise-from-mu}.   
\end{remark}

Combining Lemmas \ref{lem:weighted-regret}, \ref{lem:noise-from-g}, and \ref{lem:noise-from-mu} with equations (\ref{eq:gapsmooth}) and~(\ref{eq:threecomponentssmooth}), we establish the upper bound of the function value gap as follows.

\begin{proposition}\label{prop}
Under Assumptions \ref{asm:convex},   
 \ref{asm:lipschitz} and \ref{asm:oracle}, for all $\delta\in (0, 1)$,
 POEM (Algorithm \ref{alg:poem}) holds that  
\begin{align*}    
 f(\bar{\vx}_t)\!-\!f(\vx_\star) 
\!\leq\!\frac{16\theta_{t,\delta}(\bar{r}_t\!+\!s_0)(\sqrt{G_{t-1}}\!+\!Ld+\!L\sqrt{dt}\,)}{\sum_{k=0}^{t-1}\bar{r}_k/\bar{r}_t},
\end{align*}
with probability at least $1-\delta$, where $\theta_{t,\delta}=\log(60\log(t/\delta))$.
\end{proposition}

We then consider the lower bound for $\sum_{k=0}^{t-1}\bar{r}_k/\bar{r}_t$.
We first introduce the following basic lemma \cite{ivgi2023dog}.
\begin{lemma}[{\citet[Lemma 3.7]{ivgi2023dog}}]\label{lem:dog-tau}
We let $a_{0}, a_{1}, \dots, a_{T} $ be a positive non-decreasing sequence, then
\begin{align*}
\max _{t \leq T} \sum_{i<t} \frac{a_{i}}{a_{t}} \geq \frac{1}{{\rm e}}\bigg(\frac{T}{\log _{+}(a_{T} / a_{0})}-1\bigg),    
\end{align*}
where $\log_+(\cdot):=\log(\cdot)+1.$
\end{lemma}

Noticing that the sequence $\{\bar{r}_t\}_{t=0}^T$ is  positive and non-deceasing. 
Applying Lemma \ref{lem:dog-tau} with $a_t=\bar{r}_t$, we have
\begin{align}\label{eq:tau}
  \max_{ t\leq T}\sum_{k=0}^{t-1} \frac{\bar{r}_k}{\bar{r}_t}=   \sum_{k=0}^{\tau_T-1}\frac{\bar{r}_k}{\bar{r}_{\tau_T}}\geq \Omega\bigg(\frac{T}{\log_+({\bar{r}_{\tau_T}}/{r_\epsilon})}\bigg),
\end{align}
where $\tau_T \triangleq \argmax_{1\leq t\leq T}\sum_{k=0}^{t-1} {\bar{r}_k}/{\bar{r}_t}$.

We now consider the probability space $(\Omega_0, \fF_0, \BP)$, where $\Omega_0$ is the sample space of the Algorithm \ref{alg:poem} for given $\vx_0$ and~$r_\epsilon$, $\fF_0$ is the sigma field generated by the random sequences $\{\vv_t\}_{t=0}^T$ and $\{\vxi_t\}_{t=0}^T$, 
and $\BP:\fF_0\to[0,1]$ assigns probabilities to events in $\fF_0$.
Next, we define
\begin{align*}
   \Omega_\delta\triangleq \bigg\{\omega\in \Omega_0:\forall t\leq T,~ \bigg|\sum_{k=0}^{t-1}\bar{r}_k\inner{\vDelta_k}{\vx_k-\vx_\star}\bigg|< b_t\bigg\}.
\end{align*}
Lemma \ref{lem:noise-from-g} means
$\BP(\Omega_\delta) \geq 1-\delta$,
for all $\delta\in(0,1)$, which means $\Omega_\delta\to \Omega_0$ in probability as $\delta\to 0$. 
We then define $\fF_\delta\triangleq \{A: A\subset \Omega_\delta\} \cap \fF_0$, which is a sigma field satisfying $\fF_\delta\subset \fF_0$.
Additionally, the Lipschitz continuity and the upper bound of the domain implies $\BE|f(\bar{\vx}_{\tau_T})-f(\vx_\star)|\leq LD_\fX < \infty$.
Therefore, we can conclude the conditional expectation $\BE[f(\bar{\vx}_{\tau_T})-f(\vx_\star)\,|\,\fF_\delta]$ exists and is unique \citep[Chapter 4.1]{durrett2019probability}. 

We then combine Lemma \ref{lem:upper-bound-for-g}, Proposition \ref{prop}, and equation~(\ref{eq:tau}) to achieve our main result as follows. 
\begin{theorem}\label{thm}
Under Assumptions \ref{asm:domain}, \ref{asm:convex} ,\ref{asm:lipschitz}, and \ref{asm:oracle}, for all $\delta\in (0, 1)$, POEM (Algorithm \ref{alg:poem}) with $\vx_0\in\fX$ and $r_\epsilon\in(0,D_\fX]$ holds that 
\begin{align*}
 \BE[f(\bar{\vx}_{\tau_T})-f(\vx_\star)\,|\,\fF_\delta] 
\leq  \fO \bigg(\bigg(\frac{d}{T}+\frac{\sqrt{d}}{\sqrt{T}}\bigg)\theta_{T,\delta}LD_\fX\log_+\bigg(\frac{D_\fX}{r_\epsilon}\bigg)\bigg),
\end{align*}
with probability at least $1-\delta$, where $\theta_{T, \delta}\triangleq\log(60\log(6T/\delta))$.

By hiding the logarithmic factors by notation $\tilde\fO(\cdot)$, the SZO complexity to find an conditional expected $\epsilon$-suboptimal solution with probability $1-\delta$ is 
\begin{align*}
    \tilde\fO\bigg(\frac{dL^2 D_\fX^2}{\epsilon^2}\bigg).
\end{align*}
\end{theorem}

The SZO complexity provided by Theorem \ref{thm} matches the lower bound for stochastic zeroth-order optimization established by \citet{duchi2015optimal}.

\begin{remark}
As $\delta$ approaches 0, we have $\Omega_\delta\to\Omega_0$ and $\fF_\delta\to\fF_0$, which implies $\BE[f(\bar{\vx}_{\tau_T})-f(\vx_\star)\,|\,\fF_\delta]$ approaches  $\BE[f(\bar{\vx}_{\tau_T})-f(\vx_\star)\,|\,\fF_0]=f(\bar{\vx}_{\tau_T})-f(\vx_\star)$.  
\end{remark}

\section{Results for Unbounded Domains}\label{sec:unbounded}

In this section, we extend our method to solve the stochastic convex optimization problem such that the domain may be unbounded. 
We relax Assumption \ref{asm:domain} as follows.

\begin{assumption}\label{asm:unbounded_domain}
    The domain $\fX\subseteq\BR^d$ is closed and convex. Additionally, there exists some point $\vx_\star\in\fX$ such that $f(\vx_\star)=\min_{\vx\in\fX} f(\vx)$.
\end{assumption}

\begin{remark}    
The upper bound of $\BE[f(\bar{\vx}_{\tau_T})-f(\vx_\star)\,|\,\fF_\delta]$ in our analysis for the bounded domain (Theorem \ref{thm}) contains the term $\log_+({D_\fX}/{r_\epsilon})$ with $r_\epsilon\in(0,D_\fX]$, which may be invalid after we have relaxed
Assumption \ref{asm:domain} into Assumption~\ref{asm:unbounded_domain} 
since the diameter $D_\fX$ is possibly infinite.
\end{remark}

In the remainder of this section, we first modify POEM by introducing an underestimate of distance $\norm{\vx_0-\vx_\star}$ and an overestimate of the Lipschitz parameter $L$ to solve the problem without bounded domain assumption, 
then we show such estimates cannot be avoided for the unbounded setting.
The detailed proofs for results in this section are deferred to Appendix \ref{apx:unbounded}.

We introduce the quantity
\begin{align}\label{eq:Gprime}
 G_t^\prime=8^4 \theta_{T,\delta}\log_+^2(t+2) (G_{t-1}+16\theta_{T,\delta}d^2\bar{L}^2),
\end{align}
where 
$\theta_{T, \delta}\triangleq\log (60 \log(6 T/\delta))$ and $G_{t}\triangleq\sum_{k=0}^{t}\norm{\vg_k}^2$ follow notations in Section~\ref{sec:POEM} and 
$\bar{L}$ is an overestimate of the Lipschitz constant $L$ such that $\bar{L}\geq L$.

For the unbounded domain, we modify our POEM (Algorithm \ref{alg:poem}) by replacing the 
step size and the smoothing parameter with
\begin{align}\label{eq:mu-unbounded}
\eta_t=\frac{\bar{r}_t}{\sqrt{G_t^\prime}}
\qquad\text{and}\qquad
\mu_t = \frac{d\bar{r}_t}{(t+1)^2},
\end{align}
where $r_t\triangleq\norm{\vx_t-\vx_0}$
and $\bar{r}_{t}\triangleq\max _{k \leq t} r_{k} \vee r_{\epsilon}$ follow notations in Section \ref{sec:POEM} and
we let $G_{-1}=0$ for equation (\ref{eq:Gprime}) in the case of $t=0$.
We can observe that $T$ and $\delta$ only affect the logarithmic term in $G_t^\prime$.
Furthermore, the term $16\theta_{T,\delta}d^2\bar{L}^2$ in equation~(\ref{eq:Gprime}) is relatively smaller than  the term $G_{t-1}$ for large $t$.

We now provide the complexity analysis for the modified parameters setting (\ref{eq:mu-unbounded}). 
Note that the term $\bar r_t$ cannot be simply controlled by $D_\fX$ like the bounded domain setting.
Alternatively, we target to show that  $\bar{r}_t=\fO(s_0)$, which implies that $\vx_t$ remains close to $\vx_0$ and $\vx_\star$.
Based on the iteration
$\vx_{k+1}=\Pi_\fX (\vx_k-\eta_k \vg_k)$, 
we have
\begin{align*}
   \norm{\vx_{k+1}-\vx_\star}^2 \leq \norm{\vx_k-\vx_\star -\eta_k\vg_k}^2,
\end{align*}
Rearranging above inequality based on the definition of $s_k$ in Section \ref{sec:POEM}, we have
\begin{align*}
    s_{k+1}^2-s_k^2 \leq &\eta_k^2\norm{\vg_k}^2+2\eta_k\inner{\vDelta_k}{\vx_k-\vx_\star} -2\eta_k\inner{\nabla f_{\mu_k}(\vx_k)}{\vx_k-\vx_\star},
\end{align*}
where $\vDelta_k = \nabla f_{\mu_k}(\vx_k)-\vg_k$.
Summing above inequality over $k=0,1,\dots,t-1$, we achieve
\begin{align}\label{eq:unbounded-s}
\begin{split}    
    s_t^2-s_0^2 \leq & \sum_{k=0}^{t-1}\eta_k^2\norm{\vg_k}^2+2\sum_{k=0}^{t-1}\eta_k\inner{\vDelta_k}{\vx_k-\vx_\star}  +2\sum_{k=0}^{t-1}\eta_k\inner{\nabla f_{\mu_k}(\vx_k)}{\vx_\star-\vx_k}.
\end{split}   
\end{align}
Therefore, we can upper bound $\bar{r}_t$ by controlling each of the three terms on the right-hand side of inequality (\ref{eq:unbounded-s}). 
Note that the last term in equation (\ref{eq:unbounded-s}) does not appear in the analysis of first-order methods like DoG \cite{ivgi2023dog}.
Following the analysis in Appendix \ref{apx:unbounded-prop1}, we establish the follow upper bound for $\bar{r}_t$.

\begin{proposition}\label{prop:unbounded-1}
    For all $\delta\in(0, 1)$, POEM (Algorithm~\ref{alg:poem}) by replacing settings with $\eta_t={\bar{r}_t}/{\sqrt{G_t^\prime}}$, $\mu_t = {d\bar{r}_t}/{(t+1)^2}$, 
    and $r_\epsilon\in (0, 3s_0]$ holds that $\BP(\bar{r}_T>3s_0)\leq \delta$.
\end{proposition}

We consider the probability space $(\tilde\Omega_0, \tilde\fF_0, \tilde\BP)$, where $\tilde\Omega_0$ is the sample space of the Algorithm~\ref{alg:poem} with the modified settings used in Proposition \ref{prop:unbounded-1} for given $\vx_0$ and~$r_\epsilon$, $\Tilde{\fF}_0$ is the sigma field generated by the random sequences $\{\vv_t\}_{t=0}^T$ and $\{\vxi_t\}_{t=0}^T$, 
and $\tilde\BP:\tilde\fF_0\to[0,1]$ assigns probabilities to events in $\tilde\fF_0$.

Following the settings of Proposition \ref{prop:unbounded-1},
we define the set
\begin{align*}
     \Tilde{\Omega}_\delta \triangleq \bigg\{\omega\in \tilde\Omega_0: \forall t\leq T,  \bigg|\sum_{k=0}^{t-1}\Tilde{\eta}_k\inner{\vDelta_k}{\vx_k-\vx_\star}\bigg|\leq s_0^2\bigg\},
\end{align*}
where $\Tilde{\eta}_t=\eta_t \BI(t< \zeta)$ and $\zeta \triangleq \min\{t\in \BN \mid \bar{r}_t>3s_0\}$.

The derivation of Proposition \ref{prop:unbounded-1} (Appendix \ref{apx:unbounded-prop1}) indicates that if $r_\epsilon\leq 3s_0$, 
then $\bar{r}_T\leq 3s_0$ for all $\omega\in \Tilde{\Omega}_\delta$. 
Moreover, it holds $\tilde\BP(\Tilde{\Omega}_\delta)\geq 1-\delta$.

Similar to Proposition \ref{prop}, we can also establish the upper bound of the value gap for unbounded setting as follows (see Appendix \ref{apx:unbounded-prop2} for the proof).
\begin{proposition}\label{prop:unbounded-2}
Under Assumptions \ref{asm:convex},   
 \ref{asm:lipschitz} and \ref{asm:oracle}, for all $\delta\in (0, 1)$,
POEM (Algorithm~\ref{alg:poem}) with the modified \mbox{settings} used in Proposition \ref{prop:unbounded-1} holds that  
\begin{align}\label{eq:unbounded-gap}
 f(\bar{\vx}_t)-f(\vx_\star) 
\leq  \frac{20\theta_{t,\delta}(\bar{r}_t+s_0)(\sqrt{G_{t-1}^\prime}+ Ld)}{\sum_{k=0}^{t-1}\bar{r}_k/\bar{r}_t},
\end{align}
with probability $1-\delta$, where $\theta_{t,\delta}=\log(60\log(t/\delta))$.
\end{proposition}

Following the setting of Proposition \ref{prop:unbounded-2}, we define the set
\begin{align*}
\hat{\Omega}_\delta \triangleq \bigg\{\omega\in {\tilde\Omega}_0: \forall t\leq T,  \bigg|\sum_{k=0}^{t-1}\bar{r}_k\inner{\vDelta_k}{\vx_k-\vx_\star}\bigg|<b_t\bigg\},
\end{align*}
where $b_t=8\bar{r}_{t-1}\bar{s}_{t-1}\sqrt{\theta_{t, \delta} G_{t-1}+4L^2d^2\theta_{t,\delta}^{2}}$.
The derivation of Proposition \ref{prop:unbounded-2} indicates that inequality (\ref{eq:unbounded-gap}) holds for all $\omega\in \hat\Omega_\delta$. 
Moreover, it holds that $\tilde{\BP}(\hat{\Omega}_\delta)\geq 1-\delta$.

We define $\Tilde{F}_\delta \triangleq \{A:A\subset \Tilde{\Omega}_\delta\cap \hat\Omega_\delta\}\cap \Tilde{\fF}_0$, which is a sigma filed such that $\Tilde{F}_\delta\subset \Tilde{\fF}_0$.
Since our probability space is defined on the algorithm with finite $T$, we can conclude the conditional expectation $\BE[f(\bar{\vx}_{\tau_T}) - f(\vx_\star) \,|\, \Tilde{\fF}_\delta]$ exists and is unique even if the domain $D_\fX$ is unbounded. (Please see Appendix \ref{apx:thm2} for the detailed derivation.)

We then combine Propositions \ref{prop:unbounded-1} and \ref{prop:unbounded-2} to achieve the convergence result without bounded domain assumption.
\begin{theorem}\label{thm2}
Under Assumptions \ref{asm:convex}, \ref{asm:lipschitz},  \ref{asm:oracle}, and \ref{asm:unbounded_domain}, 
for all $\delta\in (0, {1}/{2})$, POEM (Algorithm \ref{alg:poem}) with the modified settings used in Proposition \ref{prop:unbounded-1} holds that 
\begin{align*}
\BE[f(\bar{\vx}_{\tau_T})-f(\vx_\star)\,|\,\Tilde{\fF}_\delta] 
\leq \fO \bigg(\bigg(\frac{d(L+\bar{L})}{T}+\frac{\sqrt{d}L}{\sqrt{T}}\bigg)\alpha_{T,\delta} s_0\log_+\bigg(\frac{s_0}{r_\epsilon}\bigg)\bigg)
\end{align*}
with probability at least $1-2\delta$, where $s_0= \norm{\vx_0-\vx_\star}$ and  $\alpha_{T,\delta}\triangleq \log_+(T+1)\log(60\log(T/\delta))$.
\end{theorem}

By taking $\bar{L}=L$, the upper bound of the expected function value gap shown in Theorem \ref{thm2} becomes
\begin{align*}
    \fO \bigg(\bigg(\frac{d}{T}+\frac{\sqrt{d}}{\sqrt{T}}\bigg)\alpha_{T,\delta} Ls_0\log_+\bigg(\frac{s_0}{r_\epsilon}\bigg)\bigg).
\end{align*}
This indicates the SZO complexity of $\tilde\fO(dL^2s_0^2/\epsilon^2)$
for finding an $\epsilon$-suboptimal solution $\bar{\vx}_{\tau_T}$, which is sharper than the SZO complexity of $\fO(d^2L^2s_0^2/\epsilon^2)$ provided by \citet{nesterov2017random}.

However, the settings of Theorem \ref{thm2} (also Proposition \ref{prop:unbounded-1}) requires $r_\epsilon\in (0, 3s_0]$ and the value of $s_0=\norm{\vx_0-\vx_\star}$ is unknown.
Furthermore, the first term in the upper bound provided by Theorem~\ref{thm2} contain the linear dependence on~$\bar L$.
Ideally, we desire to design an algorithm that nearly match the SZO complexity of $\tilde\fO(dL^2s_0^2/\epsilon^2)$ and only has the additional logarithmic dependence on the problem parameter estimates like $r_\epsilon$ and $\bar L$. 
Unfortunately, we can show that it is impossible to achieve such ideal parameter-free zeroth-order algorithm for the stochastic convex optimization problem without bounded domain assumptions.

We assume the stochastic zeroth-order algorithm $\fA$ accept the valid estimates $\bar{L}$, $\underline{L}$, $\bar{s}$ and $\underline{s}$ such that $\underline{L}\leq L\leq \bar{L}$ and $\underline{s}\leq s_0\leq \bar{s}$.
We then establish the following lower bound on the function value gap for our stochastic convex optimization problem.

\begin{theorem}\label{thm3}
For all polylogarithmic function $\theta: \BR^{4} \rightarrow \BR$, $d\in\BN$, and stochastic zeroth-order algorithm $\fA$ with SZO satisfying Assumption \ref{asm:oracle} and valid estimates  $\underline{L}$, $\bar{L}$, $\underline{s}$, and~$\bar{s}$, 
there exists an $L$-Lipschitz and convex function $f$ defined on $\BR^d$ such that algorithm $\fA$ with SZO call of $T\geq 2$ and some initial point $\vx_0\in\BR^d$ returns a point $\hat{\vx}$ satisfying
\begin{align*}
    f(\hat{\vx})-f_\star>\theta\bigg(\frac{\bar{L}}{\underline{L}}, \frac{\bar{s}}{\underline{s}}, T, d\bigg)  \frac{\sqrt{d}\,L s_{0}}{\sqrt{T}}
\end{align*}
with probability at least $1/{\rm e}$.  
\end{theorem}

\begin{remark}
In a recent work, \citet{khaled2024tuning} showed the impossibility of the ideal parameter-free algorithm for stochastic first-order optimization by construction the hard instance of one dimensional function.
In contrast, the lower bound for zeroth-order optimization shown in Theorem \ref{thm3} needs to additionally consider the dependence on dimension, which is different from the analysis in the first-order case. 
\end{remark}

\section{Numerical Experiments}

This section conducts numerical experiments to evaluate the empirical performance of proposed POEM (Algorithm \ref{alg:poem}).
We consider the stochastic optimization problem of the form
\begin{align*}
\min_{\vx\in\fX} f(\vx) \triangleq \BE_{(\va,b)}[F(\vx; \va,b)],
\end{align*}
where $F(\vx;\va,b)=\max\{0,1-b\va^{\top} \vx\}$,  
$(\va,b)\!\in\!{\BR^{d}\!\times\!\{\pm1\}}$ 
is uniformly sampled from the binary classification dataset $\{(\va_i,b_i)\}_{i=1}^n$, and $\fX=\{\vx\in\BR^d:\norm{\vx}\leq R\}$ with the
radius $R=1$.
We perform our experiments on datasets \cite{chang2011libsvm} 
``mushrooms'' ($d=112$, $n=8124$), ``a9a" ($d=123$, $n=32,561$) and , ``w8a'' ($d=300$, \mbox{$n=49,749$}).
We compare the proposed POEM with existing stochastic zeroth-order algorithms Two-Point Gradient Estimates (TPGE) method \cite{duchi2015optimal} and Two-Point Bandit Convex Optimization (TPBCO) method \cite{shamir2017optimal}.

We present the comparison on the SZO complexity against the function values in Figure \ref{fig:szo}. 
For POEM, we simply set the inital movment as $r_\epsilon = 10^{-2}$. 
We run the baseline methods TPGE and TPBCO with the parameter settings in their theoretical analysis (refer to TPGE-T and TPBCO-T) and well-tuned step sizes (refer to TPGE-E and \mbox{TPBCO-E}), respectively.
We observe the POEM converges faster than TPGE-T and TPBCO-T.
In addition, results of POEM and baseline methods with well-tuned stepsizes are comparable. 

We also study the impact of parameter settings in practice.
Specifically, we present the objective function value at the last iteration $(T=10^6)$ for all algorithms with different parameter settings in Figure~\ref{fig:tune},
where we tune the initial movement $r_\epsilon$ in our POEM and the term $1/L$ in expressions of stepsizes in baseline methods from $\{10^{-7},10^{-6},\dots,10^2\}$.
It is clear that our POEM is more robust to the parameter settings than other methods.
More importantly, the initial movement in POEM almost does not affect the $f(\vx_T)$ when we take $r_\epsilon\leq R=1$,
which support our theory that $r_\epsilon$ only affect the logarithmic term in the complexity bound if it is no larger than the diameter of the domain (Theorem \ref{thm}).
We also show the change of the step sizes with different initial movement $r_\epsilon$ for POEM, which shows the step sizes with different settings tend to the same with iterations.

\begin{figure*}[!ht]
\centering
\begin{tabular}{ccc}
     \includegraphics[scale=0.23]{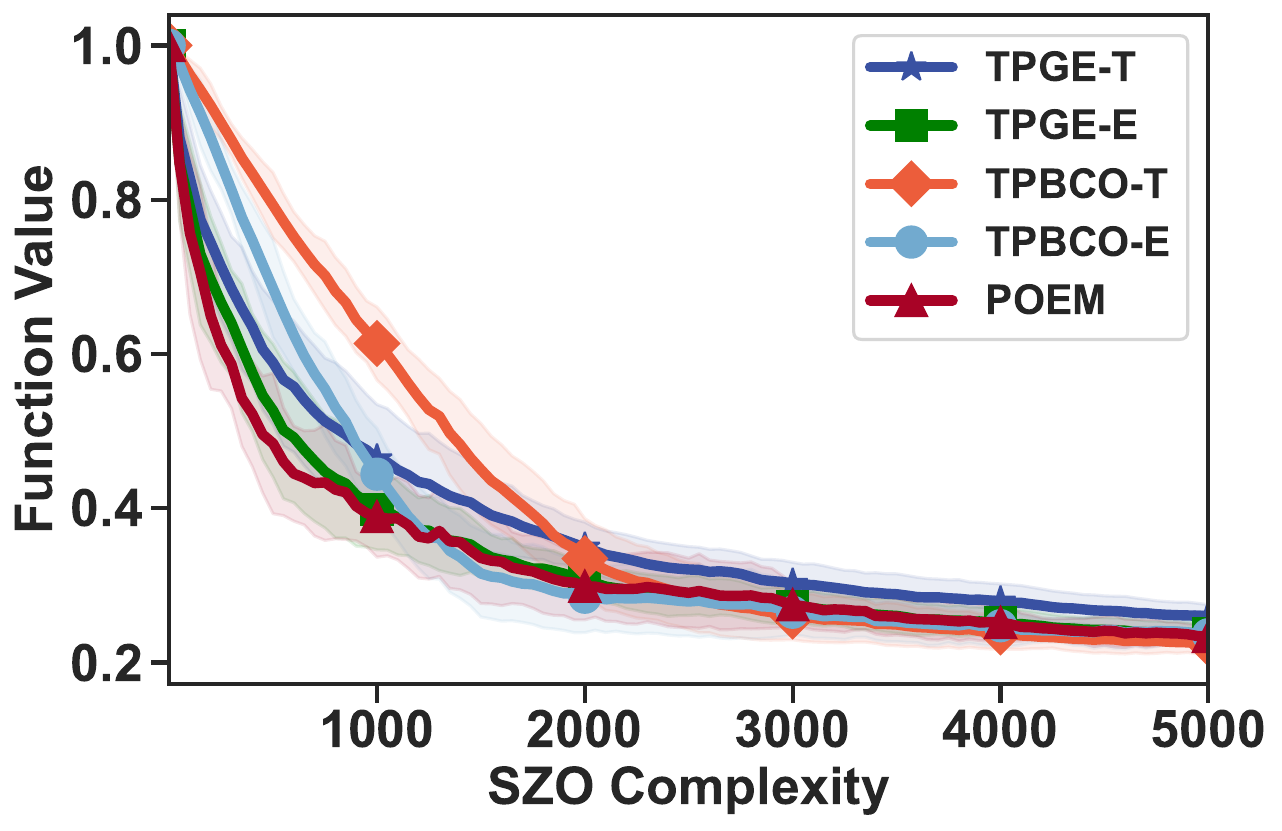} 
     &  \includegraphics[scale=0.23]{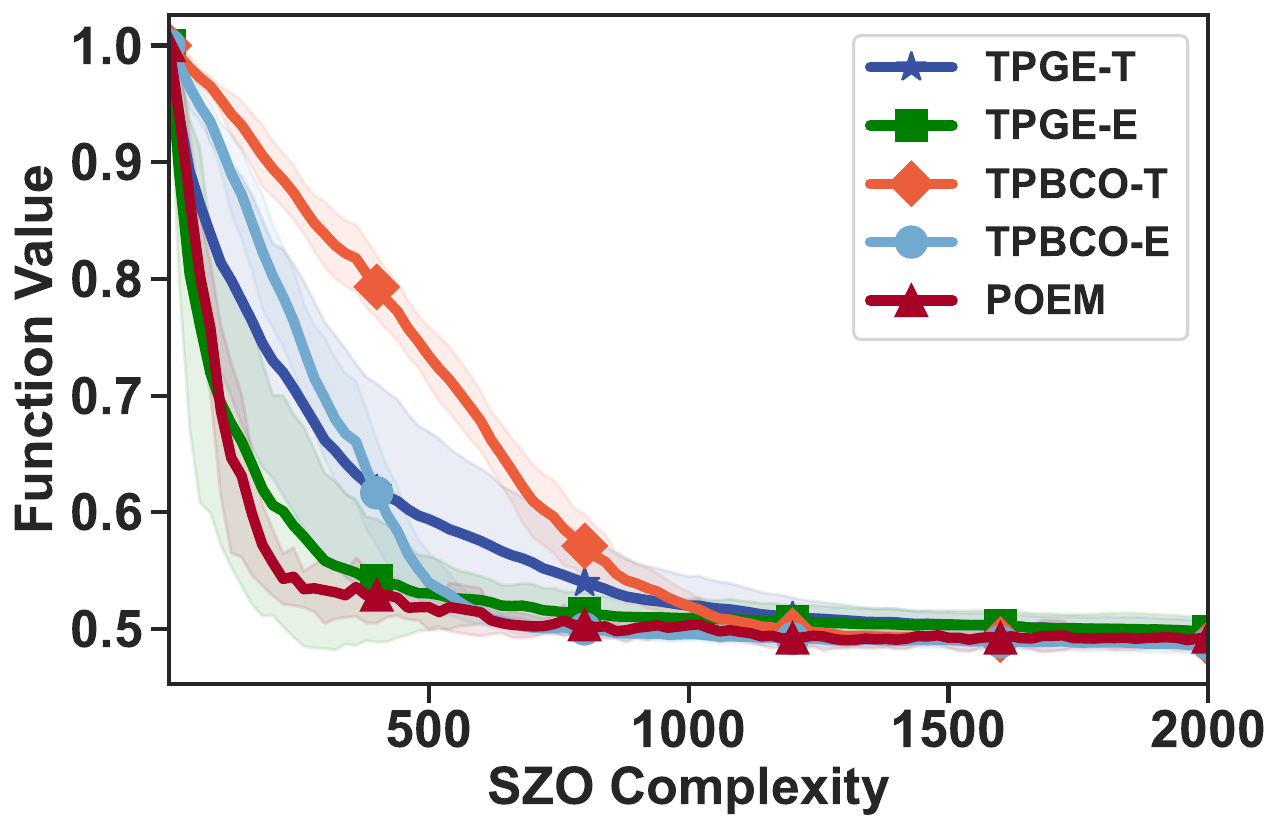}
     &  \includegraphics[scale=0.23]{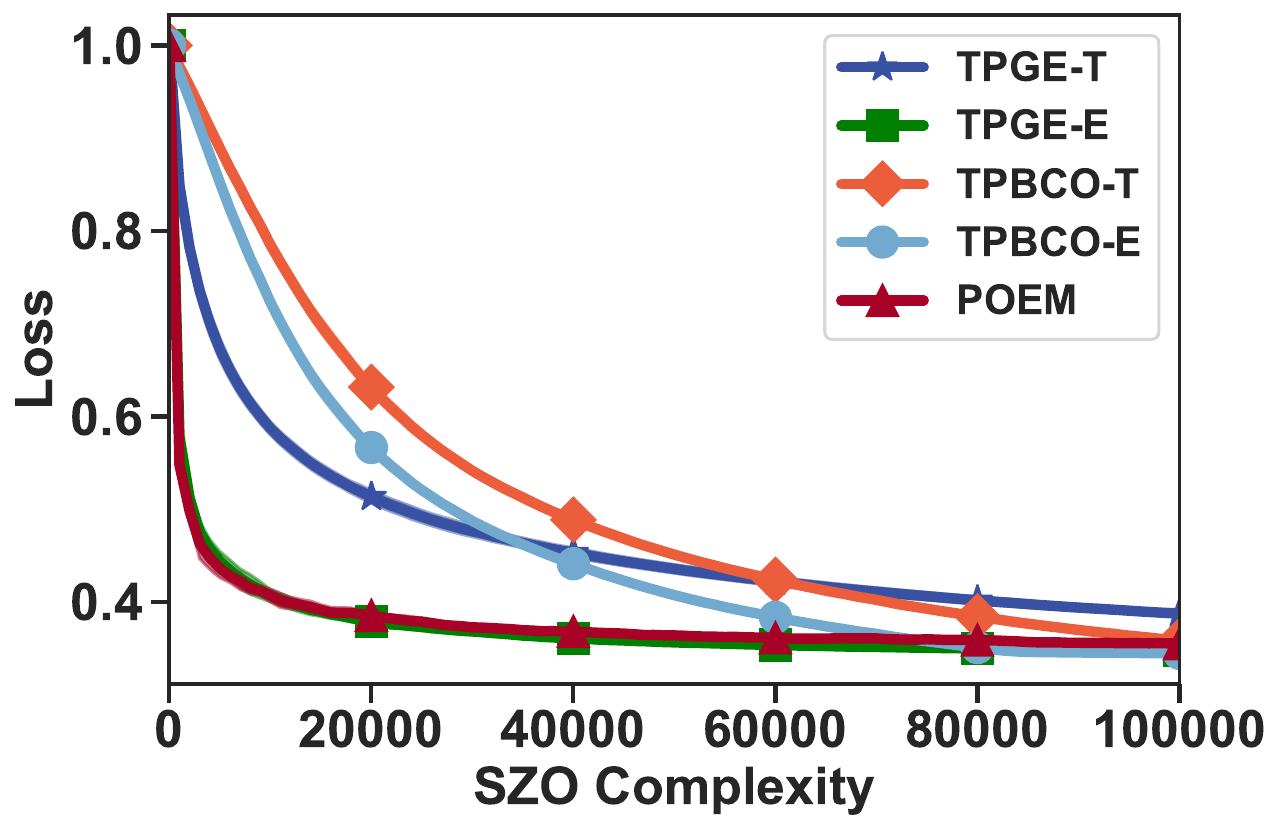} \\[-0.1cm]
     (a) \footnotesize mushrooms & 
     (b) \footnotesize a9a & 
     (c) \footnotesize w8a
\end{tabular}  
\caption{The comparison on the SZO complexity against the function value during the iterations.}\label{fig:szo} \vskip0.15cm
\end{figure*}

\begin{figure*}[!ht]
\centering
\begin{tabular}{ccc}
     \includegraphics[scale=0.23]{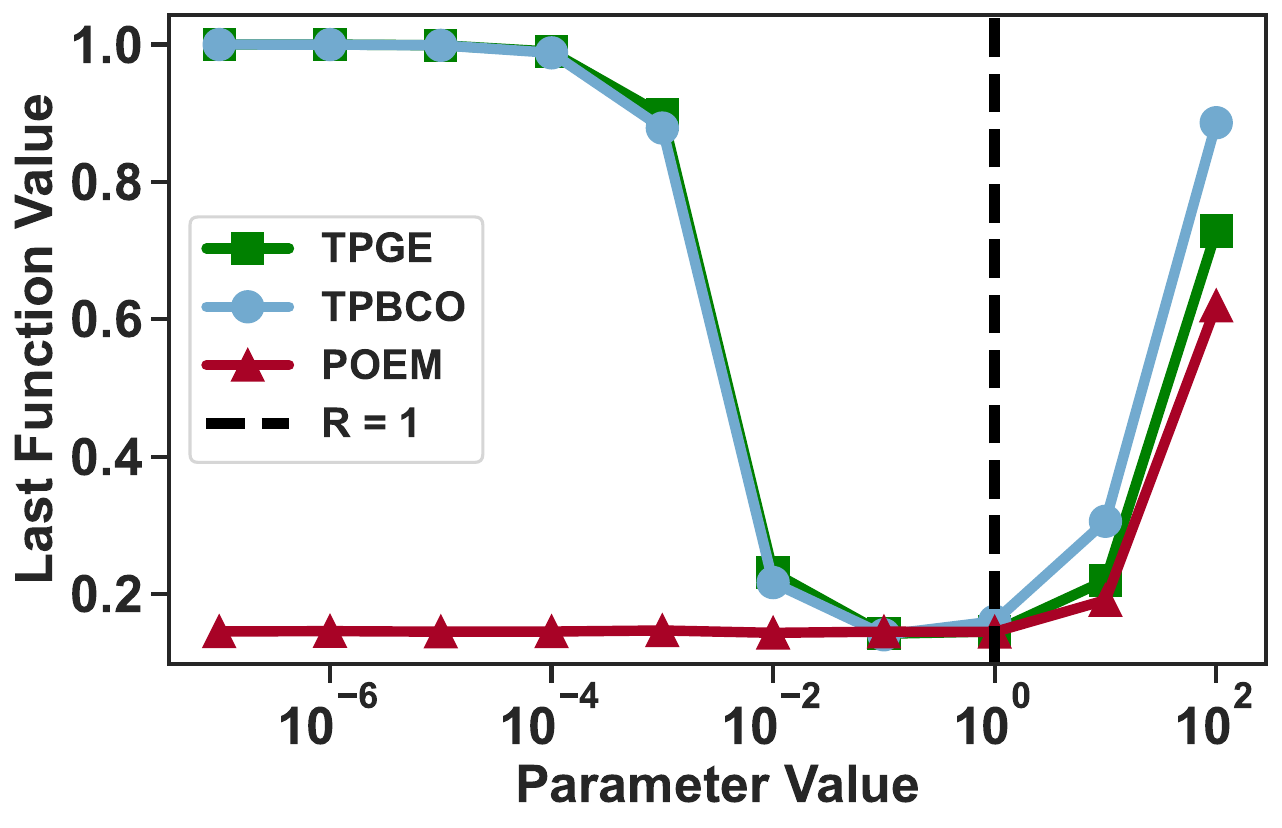} 
     &  \includegraphics[scale=0.23]{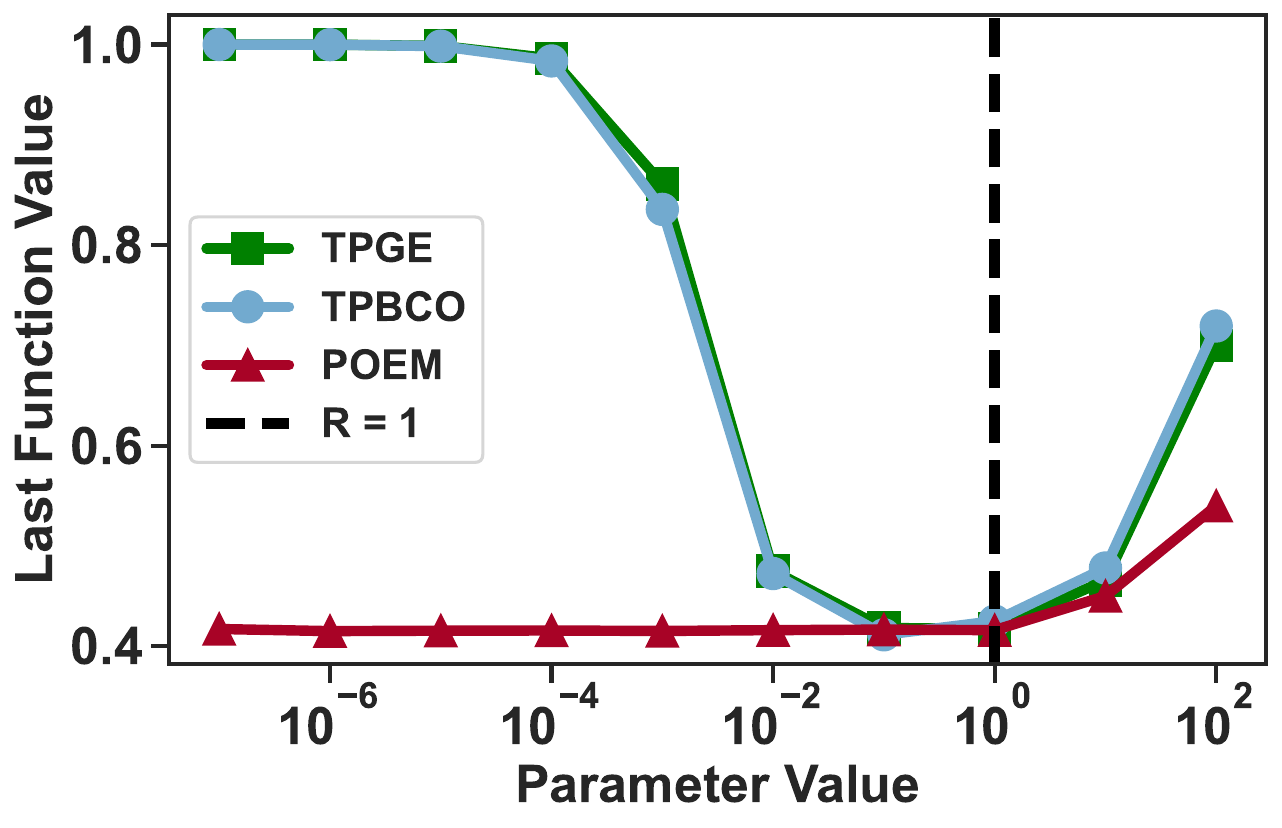}
     &  \includegraphics[scale=0.23]{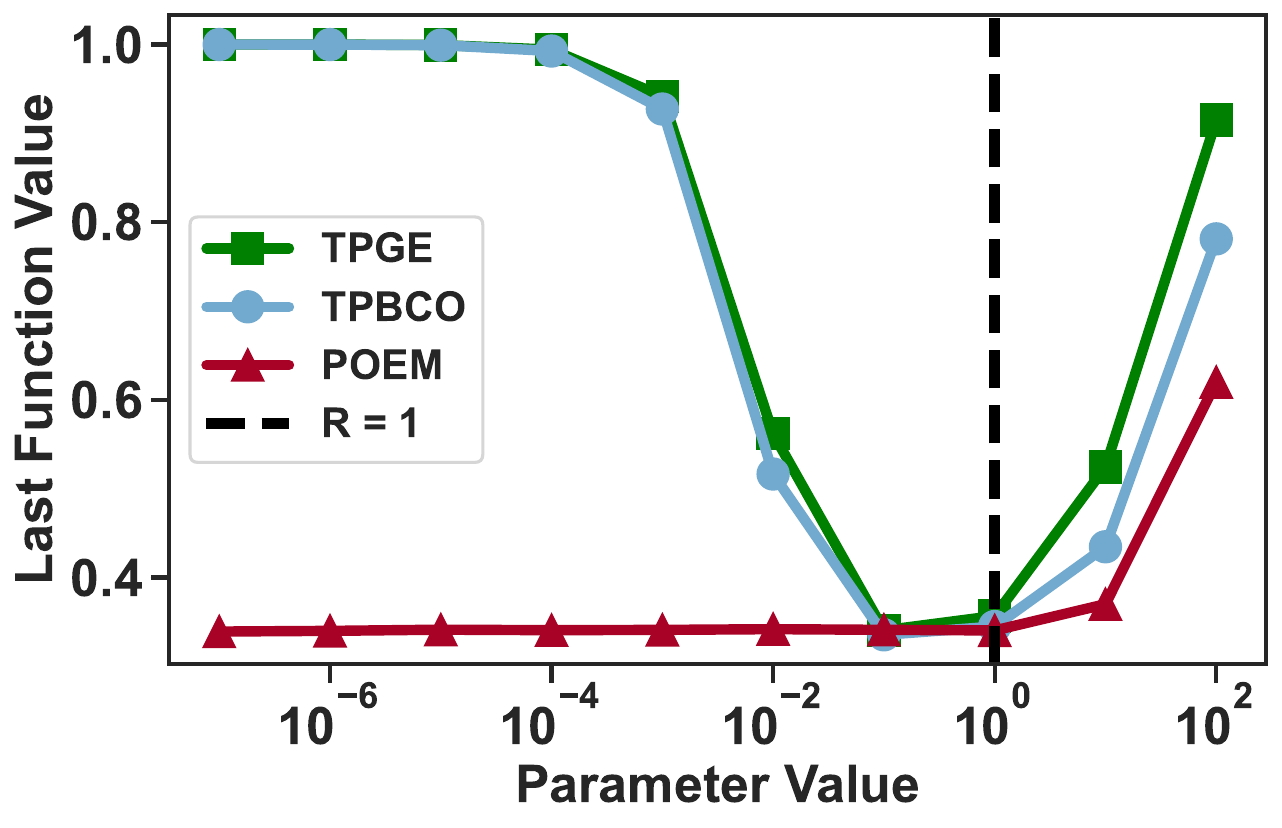} \\[-0.1cm]
     (a) \footnotesize mushrooms & 
     (b) \footnotesize a9a & 
     (c) \footnotesize w8a
\end{tabular} 
\caption{The comparison on parameter settings ($r_\epsilon$ for POEM and $1/L$ for other methods) against $f(\vx_T)$.}\label{fig:tune}  \vskip0.15cm
\end{figure*} 

\begin{figure*}[!ht]
\centering
\begin{tabular}{ccc}
     \includegraphics[scale=0.23]{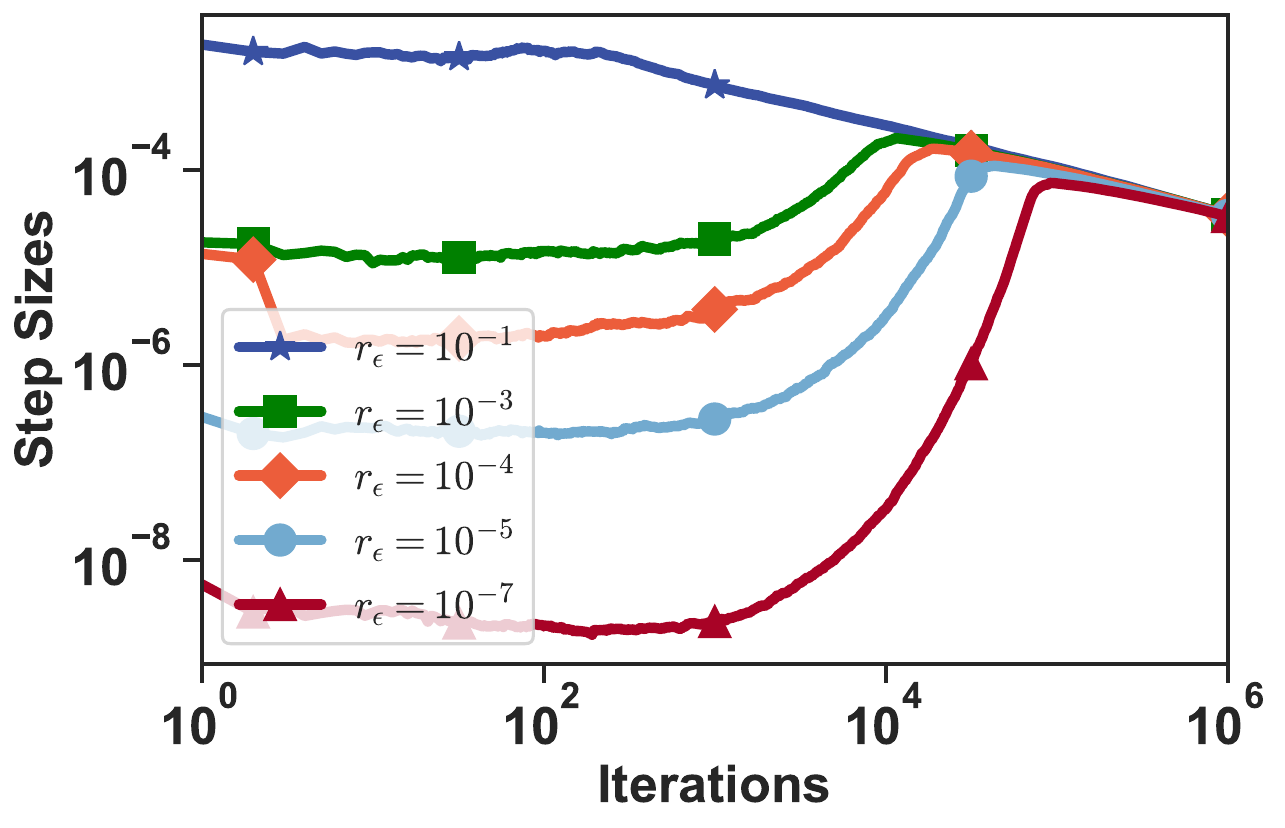} 
     &  \includegraphics[scale=0.23]{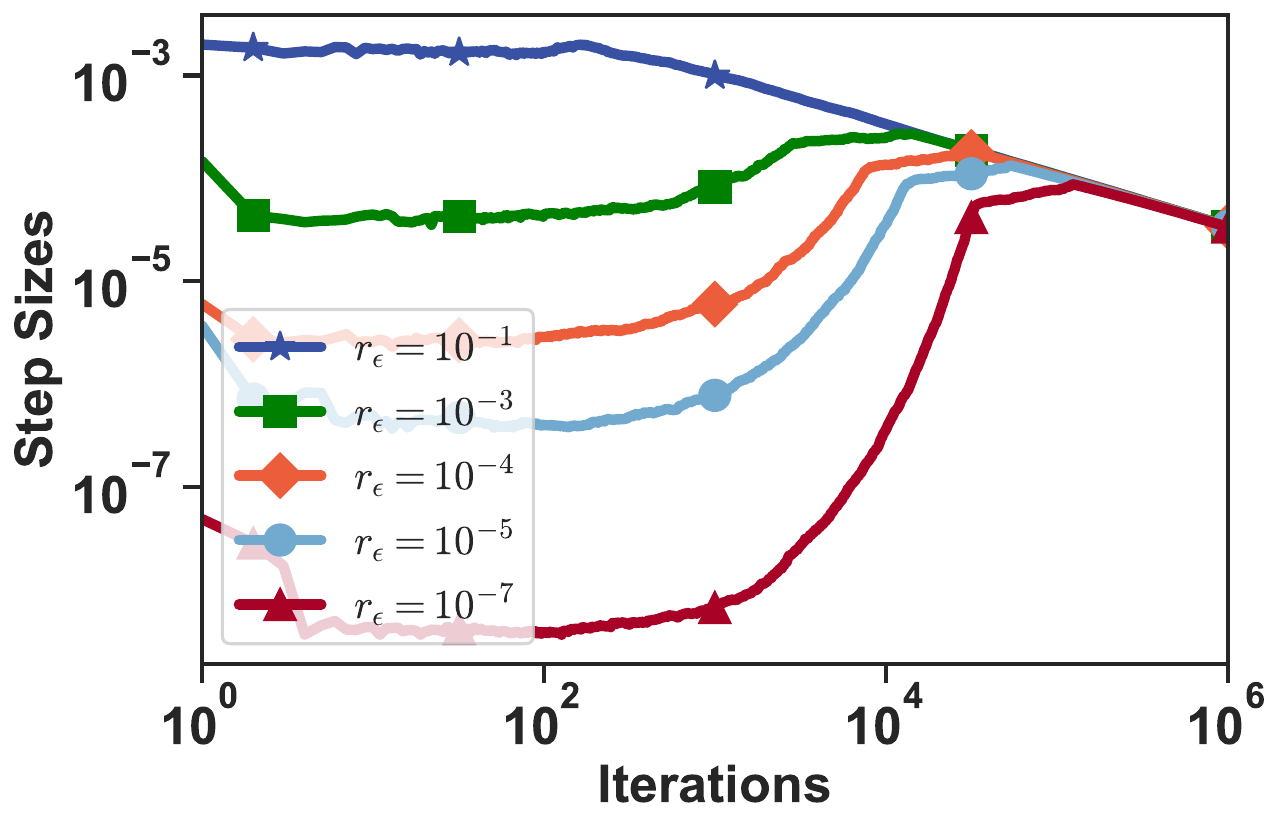}
     &  \includegraphics[scale=0.23]{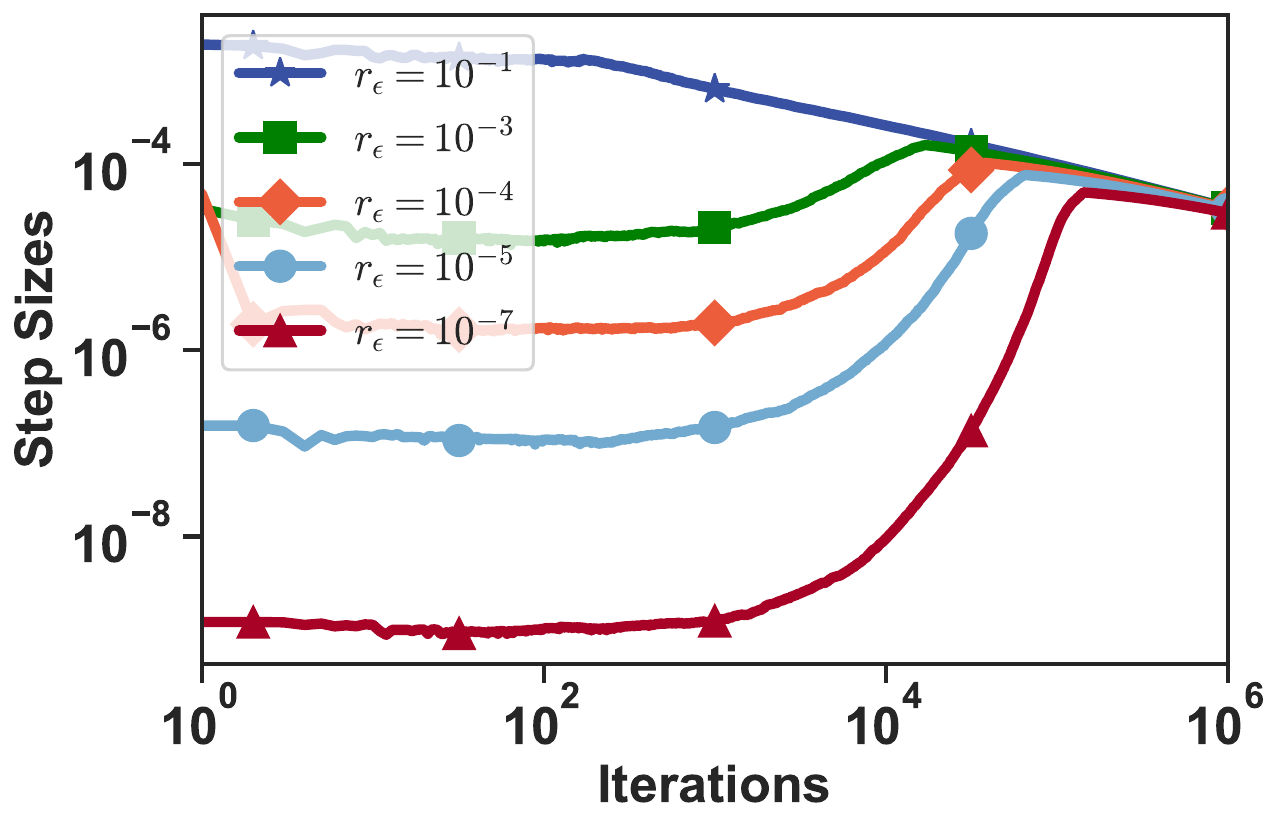} \\[-0.1cm]
     (a) \footnotesize mushrooms & 
     (b) \footnotesize a9a & 
     (c) \footnotesize w8a
\end{tabular}  
\caption{The change of the step size with difference $r_\epsilon$ for POEM.}\label{fig:stepsize}  \vskip0.15cm
\end{figure*} 

\section{Conclusion}
In this paper, we propose a novel zeroth-order optimization algorithm POEM for stochastic convex optimization, which schedules both the step size and the smoothing parameter during iterations.
We show that POEM achieves the near-optimal stochastic zeroth-order oracle complexity and its initialization only affects the convergence rates by a logarithmic factor for problems with bounded domain. 
We then extend POEM to solve the problem without the bounded domain assumption and provide the lower bound to show the limitation of stochastic zeroth-order algorithms in the unbounded setting.
Our numerical experiments further validate the efficiency of POEM in practice.

In future work, we are interested in extending the ideas of POEM to a wider range of scenarios such as the zeroth-order optimization for minimax and bilevel problems.
It is also possible to design the parameter-free zeroth-order methods for the finite-sum optimization problem.

\bibliographystyle{plainnat}
\bibliography{reference}

\newpage
\appendix

\section{Some Basic Results}

We first present some basic lemmas.

\begin{lemma}\label{lem:fact-convex}
Suppose the function $\phi:\fD\to \BR$ is differentiable, then it is convex if and only if for any $\vx,\vy\in \fD$, we have
\begin{align*}
   \phi(\vy)\geq \phi(\vx)+\inner{\nabla \phi(\vx)}{\vy-\vx},
\end{align*}
where $\nabla \phi(\vx)$ denotes the gradient of $\phi(\cdot)$ at $\vx$.
\end{lemma}

\begin{lemma}\label{lem:fact-sphere}
Suppose $\vv\sim \BU(\BS^{d-1})$, then the covariance matrix of $\vv$ is given by
\begin{align*}
    \mathbb{E}_{\vv\sim \BU(\BS^{d-1})}[\vv\vv^\top]=\frac{1}{d}\cdot\mathbf{I}_d,
\end{align*}
where $\mathbf{I}_d$ denotes the identity matrix with order $d$.
\end{lemma}

\begin{lemma}[{\citet[Lemma C.3]{ivgi2023dog}}]\label{lem:dog-unbounded-1}
 Let $a_{-1},a_0,\dots, a_{t}$   be a nondecreasing sequence of nonnegative numbers, then the following inequality holds
 \begin{align*}
     \sum_{k=0}^t \frac{a_k-a_{k-1}}{a_k\log_+^2(a_k/a_{-1})}\leq 1.
 \end{align*}
\end{lemma}

\begin{lemma}[{\citet[Lemma 9]{shamir2017optimal}}]\label{lem:shamir-concentration}
Suppose $\vv\sim \BU(\BS^{d-1})$. Then, for any function $h:\BR^d\to\BR$ which is $L$-Lipschitz with respect to the $2$-norm, it holds that 
\begin{align*}
    \BP(|h(\vv)-\BE_{\vv\sim \BU(\BS^{d-1})}[h(\vv)]|\geq t)\leq 2\exp\bigg(-\frac{cdt^2}{L^2}\,\bigg),
\end{align*}
where $c$ is a numerical constant.
\end{lemma}

\begin{lemma}[{\citet[Lemma D.2]{ivgi2023dog}}] \label{lem:dog-martingale}
     Let  $S $ be the set of nonnegative and nondecreasing sequences. Let  $c>0 $ and let  $X_{t}$  be a martingale difference sequence adapted to  $\mathcal{F}_{t} $ such that  $|X_{t}| \leq c$  with probability 1 for all  $t$ . 
     Then, for all  $\delta \in(0,1) $ and  $\hat{X}_{t} \in \mathcal{F}_{t-1}$  such that  $|\hat{X}_{t}| \leq c $ with probability 1,
\begin{align*}
    \BP\bigg(\exists t\leq T,\exists \{y_k\}_{k=0}^\infty\in S:\bigg|\sum_{k=0}^{t-1}y_kX_k\bigg|\geq b_t\bigg)\leq \delta,
\end{align*}
where $b_t=8y_t\sqrt{\theta_{t,\delta}\sum_{k=0}^{t-1}(X_k-\hat{X}_k)^2+c^2\theta_{t,\delta}^2}\,$ and $\theta_{t,\delta}=\log(60\log(6t/\delta))$.
\end{lemma}

\begin{lemma}[{\citet[Corollary 1]{carmon2022making}}]\label{lem:dog-unbounded-2}
Let $c > 0$ and $X_t$ be a martingale difference sequence adapted
to $\fF_t$ such that $|X_t| \leq c$ with probability $1$ for all $t$. Then, for all $\delta\in (0,1)$ , and $\hat{X}_t\in \fF_{t-1}$ such that $|\hat{X}_t|\leq c$ with probability 1, the following inequality holds
 \begin{align*}
     \BP\bigg(\exists t\leq T: \bigg|\sum_{k=1}^{t}X_k\bigg|>4\sqrt{\theta_{t,\delta}\sum_{k=1}^{t}(X_k-\hat{X}_k)^2+c^2\theta_{t,\delta}^2}\bigg)\leq \delta.
 \end{align*}
\end{lemma}

\section{The Proofs for Section \ref{sec:complexity}}\label{apx:complexity}

We provide the detailed proofs for results under the bounded domain assumption.
 
\subsection{Proof of Lemma \ref{lem:upper-bound-for-g}}\label{apx:upper-bound-for-g}

\begin{proof}[Proof of Lemma \ref{lem:upper-bound-for-g}]
For convenience, we drop the subscripts. 
Recall that $\vg$ is defined as
\begin{align*}
\vg(\vx, \mu; \vv,\vxi) = \frac{d}{2\mu}(F(\vx+\mu\vv;\vxi)-F(\vx-\mu\vv;\vxi))\vv,\quad \text{where} \quad \vv\sim \BU(\BS^{d-1}).
\end{align*}
By Assumption \ref{asm:lipschitz}, the function $F(\vx;\xi)$ is $L$-Lispchitz almost surely.
Thus, we have
\begin{align*}
    \norm{\vg} = \frac{d}{2\mu}|F(\vx+\mu\vv;\vxi)-F(\vx-\mu\vv;\vxi)|\norm{\vv} \leq Ld \norm{\vv}^2=Ld,
\end{align*}
where the last equality uses the fact $\norm{\vv}=1$. 

Next, using the definition of $\vg$ and $\norm{\vv}=1$ again, we have
\begin{align*}
    \BE_{\vv\sim \BU(\BS^{d-1})}[\norm{\vg}^2] = \frac{d^2}{4\mu^2}\cdot \BE_{\vv\sim \BU(\BS^{d-1})}[(F(\vx+\mu\vv;\vxi)-F(\vx-\mu\vv;\vxi))^2].
\end{align*}
Introducing $\alpha\in \BR$, we rewrite the expression as
\begin{align*}
     \BE_{\vv\sim \BU(\BS^{d-1})}[\norm{\vg}^2] = \frac{d^2}{4\mu^2}\cdot \BE_{\vv\sim \BU(\BS^{d-1})}[((F(\vx+\mu\vv;\vxi)-\alpha)-(F(\vx-\mu\vv;\vxi)-\alpha))^2].
\end{align*}
Using the inequality $(a-b)^2\leq 2a^2+2b^2$, this becomes
\begin{align*}
   \BE_{\vv\sim \BU(\BS^{d-1})}[\norm{\vg}^2]  \leq \frac{d^2}{2\mu^2}\cdot (\BE_{\vv\sim \BS}[(F(\vx+\mu\vv;\vxi)-\alpha)^2]+ \BE_{\vv\sim \BS}[(F(\vx-\mu\vv;\vxi)-\alpha)^2]).
\end{align*}
Since the distribution $\vv\sim \BU(\BS^{d-1})$ is symmetric around the origin,  the two terms are equal. 
Applying this property, we derive
\begin{align}\label{eq:upsd-nonsmooth}
    \BE_{\vv\sim \BU(\BS^{d-1})}[\norm{\vg}^2] \leq \frac{d^2}{\mu^2}\cdot \BE_{\vv\sim \BU(\BS^{d-1})}[(F(\vx+\mu\vv;\vxi)-\alpha)^2].
\end{align}

Define $h(\vv):= F(\vx+\mu\vv;\vxi)$.
Since $F(\vx;\xi)$ is $L$-Lispchitz w.r.t. $\vx$ under Assumption \ref{asm:lipschitz}, then we have $h(\vv)$ is $\mu L$-Lipschitz w.r.t. $\vv$. 
Applying the concentration bound from Lemma \ref{lem:shamir-concentration}, we obtain
\begin{align*}
    \BE_{\vv\sim \BU(\BS^{d-1})}[(h(\vv)-\BE_{\vv\sim \BU(\BS^{d-1})}[h(\vv)])^2] &= \int_0^\infty \BP((h(\vv)-\BE_{\vv\sim \BU(\BS^{d-1})}[h(\vv)])^2>t)\, \mathrm{d}t\\
    & = \int_0^\infty \BP(|h(\vv)-\BE_{\vv\sim \BU(\BS^{d-1})}[h(\vv)]|>\sqrt{t})\, \mathrm{d}t\\
    &\leq \int_0^\infty 2\exp(-\frac{cdt}{\mu^2L^2})\,\mathrm{d}t=\frac{2\mu^2L^2}{cd},
\end{align*}
where $c>0$ is a numerical constant. 
The first equality holds by the property \citep[Lemma 2.2.13]{durrett2019probability} that if a random variable $Y>0$ almost surely, then 
\begin{align*}
\BE[Y] = \int_0^\infty \BP(Y>y)\, \mathrm{d}y.
\end{align*}
Let $\alpha= \BE_{\vv\sim \BU(\BS^{d-1})}[h(\vv)]$.
Combining the concentration bound above and inequality (\ref{eq:upsd-nonsmooth}), we have the following inequality
\begin{align*}
   \BE_{\vv\sim \BU(\BS^{d-1})}[\norm{\vg}^2] \leq \frac{d^2}{\mu^2}\cdot \BE_{\vv\sim \BU(\BS^{d-1})}[(h(\vv)-\BE_{\vv\sim \BU(\BS^{d-1})}[h(\vv)])^2]\leq  \frac{2}{c} L^2 d.
\end{align*}
Thus, using the law of total expectation, we have $\BE[\norm{\vg}^2]=\BE[\BE_{\vv\sim \BU(\BS^{d-1})}[\norm{\vg}^2]]\leq 2L^2d/c$.
\end{proof}

\subsection{Proof of Lemma \ref{lem:noise-from-g}}\label{apx:noise-from-g}

\begin{proof}[Proof of Lemma \ref{lem:noise-from-g}]
We begin by defining the filtration as  
    $\fF_k=\sigma(\vv_i, \vxi_i, 0\leq i\leq k)$ for $k\in \BN$ and $\fF_{-1}=\{\emptyset,\Omega\}$.
Then, we define the stochastic processes $(X_k,k\in \BN)$ and $(\hat{X}_k,k\in \BN)$ as follows
\begin{align}\label{eq:def-martingale}
    X_k=\frac{1}{\bar{s}_k}\inner{\vDelta_k}{\vx_k-\vx_\star} \in \fF_k \quad\text{and}\quad \hat{X}_k = \frac{1}{\bar{s}_k}\inner{\nabla f_{\mu_k}(\vx_k)}{\vx_k-\vx_\star}\in \fF_{k-1},
\end{align}
where $\vDelta_k=\nabla f_{\mu_k}(\vx_k)-\vg_k$.
Thus,  we derive that $(X_k,k\in \BN)$ is adapted to $ (\fF_k,k\in \BN)$ and $(\hat{X}_k,k\in \BN)$ is predictable w.r.t. $ (\fF_k,k\in \BN)$.
In addition, using the fact that given $\fF_{k-1}$, the stochastic difference $ \vg_k$ is an unbiased estimator for $\nabla f_{\mu_k}(\vx_k)$, we obtain 
\begin{align*}
    \BE[X_k\,|\,\fF_{k-1}] = \frac{1}{\bar{s}_k}\cdot \BE[\inner{\nabla f_{\mu_k}(\vx_k)-\vg_k}{\vx_k-\vx_\star}\,|\,\fF_{k-1}]=0,\quad \text{where}\quad k\in\BN.
\end{align*}
The first inequality holds since $\bar{s}_k\in \fF_{k-1}$. 
This implies that $\left(X_k,\mathcal{F}_k, k\in \mathbb{N}\right)$ is a martingale difference process.

Next, from Lemma \ref{lem:upper-bound-for-g}, we have $\norm{\vg_k}\leq  Ld$, which implies that
\begin{align*}
    \norm{\nabla f_{\mu_k}(\vx_k)} = \norm{\BE_{\vv_k, \vxi_k}[\vg_k]}\leq \BE_{\vv_k,\vxi_k}[\norm{\vg_k}]\leq Ld,
\end{align*}
The first inequality because $\vg_k$ is an unbiased estimator of $\nabla f_{\mu_k}(\vx_k)$. Then, we have
\begin{align*}
    \norm{\vDelta_k}\leq \norm{\vg_k}+\norm{\nabla f_{\mu_k}(\vx_k)}\leq 2Ld,\quad \text{where} \quad k\in \mathbb{N}.
\end{align*}
From equation (\ref{eq:def-martingale}), we have $|X_k|\leq \norm{\vDelta_k}$ and $|\hat{X}_k|\leq \norm{\nabla f_{\mu_k}(\vx_k)}$.
Hence, we obtain that both $X_k$ and $\hat{X}_k$ are bounded above by $2Ld$. 

Define the sequence $ Y_{k}:=\bar{r}_{k} \bar{s}_{k}$ for $k\in \BN$, which is nonnegative and nondecreasing. 
Now,  using the concentration bound of martingale difference introduced in Lemma \ref{lem:dog-martingale} with $\delta\in (0,1)$ and $c=2Ld$,  we derive the following upper bound
 \begin{align*}
     \BP\bigg(\exists t\leq T: \bigg|\sum_{k=0}^{t-1}\bar{r}_k\inner{\vDelta_k}{\vx_k-\vx_\star}\bigg|\geq b_t\bigg) &= \BP\bigg(\exists t\leq T: \bigg|\sum_{k=0}^{t-1} Y_kX_k \bigg|\bigg)\\
     & \leq \BP\bigg(\exists t\leq T, \exists \{y_k\}_{k=0}^\infty\in S: \bigg|\sum_{k=0}^{t-1} y_kX_k\bigg|\geq b_t\bigg)\\
     & \leq \delta,
 \end{align*}
where $b_t= 8 \bar{r}_{t-1} \bar{d}_{t-1} \sqrt{\theta_{t, \delta} G_{t-1}+4L^2d^2\theta_{t, \delta}^{2}}\,$, $\theta_{t,\delta}=\log(60\log(6t/\delta)$ and $S$ is the set of non-negative and non-decreasing sequences.
\end{proof}

\subsection{Proof of Lemma \ref{lem:noise-from-mu}}\label{apx:noise-from-mu}
\begin{proof}[Proof of Lemma \ref{lem:noise-from-mu}]
Define the partial sum as $S_t:=\sum_{k=1}^t 1/\sqrt{k}$ for $t\in \BN_+$. The upper bound of $S_t$ is given by the following integral
\begin{align*}
   S_t\leq  1+\int_1^t \frac{1}{\sqrt{x}} \,\mathrm{d} x =2\sqrt{t}-1\leq 2\sqrt{t}.
\end{align*}
Based on this upper bound and the formula of $\mu_k$ shown in equation (\ref{def:smooth}), we can bound the noise as follows
\begin{align*}
\sum_{k=0}^{t-1}2L\bar{r}_k\mu_k \leq 2L\bar{r}_{t-1}\sum_{k=0}^{t-1}\mu_k = 2L\sqrt{d}\cdot\bar{r}_{t-1}S_t=4L \bar{r}_{t-1}\sqrt{dt}.
\end{align*}
\end{proof}

\subsection{Proof of Proposition \ref{prop}}\label{apx:prop}
\begin{proof}[Proof of Proposition \ref{prop}]
Combining Lemma \ref{lem:weighted-regret}, \ref{lem:noise-from-g} and \ref{lem:noise-from-mu} along with equations (\ref{eq:gapsmooth}) and (\ref{eq:threecomponentssmooth}), we obtain that, with probability at least $1-\delta$,  the upper bound for  $f(\bar{\vx}_t)-f(\vx_\star)$ is 
\begin{align*}
\frac{(2\bar{s}_t+\bar{r}_t)\sqrt{G_{t-1}} + 8\bar{s}_t\sqrt{\theta_{t,\delta}G_{t-1}+4L^2d^2\theta_{t,\delta}^2}\,+4\bar{r}_tL\sqrt{dt}}{\sum_{k=0}^{t-1}\bar{r}_k/\bar{r}_t},
\end{align*} 
where we ignore the constants and use the fact that $\bar{r}_{t-1}\leq \bar{r}_t$ and $\bar{s}_{t-1}\leq \bar{s}_t$.
Applying the inequality $\sqrt{a^2+b^2}\,\leq a+b$, the gap simplifies to
\begin{align*}
         \frac{(2\bar{s}_t+\bar{r}_t)\sqrt{G_{t-1}} + 8\bar{s}_t(\theta_{t,\delta}\sqrt{G_{t-1}}+2\theta_{t,\delta}Ld)+4\bar{r}_t L\sqrt{dt}}{\sum_{k=0}^{t-1}\bar{r}_k/\bar{r}_t}.
\end{align*}
Finally, using the triangle inequality $\bar{s}_t\leq \bar{r}_t+s_0$, the gap becomes
\begin{align*}
    16\cdot \frac{\theta_{t,\delta}(\bar{r}_t+s_0)(\sqrt{G_{t-1}}+ Ld+L\sqrt{dt}\,)}{\sum_{k=0}^{t-1}\bar{r}_k/\bar{r}_t}.
\end{align*}
\end{proof}

\subsection{Proof of Theorem \ref{thm}}\label{apx:thm}

First, we illustrate some useful properties of conditional expectation.

\begin{lemma}\label{lem:ce-inequality}
 Let $(\Omega, \fF_0, \BP)$ be a probability space, $X_1, X_2$ be two random variables on it, and $\BE[\cdot\, |\,\fF]$ be the corresponding conditional expectations where $\fF\subset \fF_0$.
Then , if $X_1\leq X_2$ on $B\in \fF$, then $\BE[X_1\,|\,\fF]\leq \BE[X_2\,|\,\fF]$ on $B$.
\end{lemma}
\begin{proof}[Proof of Lemma \ref{lem:ce-inequality}]
We follow the approach of \citet[Theorem 4.1.2]{durrett2019probability} to prove the lemma.
Given any $\epsilon>0$, we define the event $A=\{\omega\in \Omega:\BE[X_1\,|\,\fF]-\BE[X_2\,|\,\fF]\geq  \epsilon\}$, which satisfies $A\in \fF$ since both conditional expectations are $\fF$-measurable. 
Thus, we have $A\cap B \in \fF$ since $\fF$ is closed under intersections.
Then, by the definition of conditional expectation, we have
\begin{align*}
    \int_{A\cap B} \BE[X_1\,|\,\fF]-\BE[X_2\,|\,\fF] \,\mathrm{d}\BP =  \int_{A\cap B} X_1-X_2 \,\mathrm{d}\BP \leq 0.
\end{align*}
By the definition of $A$, we further have
\begin{align*}
     \int_{A\cap B} \BE[X_1\,|\,\fF]-\BE[X_2\,|\,\fF] \,\mathrm{d}\BP \geq    \int_{A\cap B} \epsilon \,\mathrm{d}\BP = \epsilon\cdot \BP(A\cap B).
\end{align*}
Combining the two inequality, we have $\BP(A\cap B)=0$, which implies that
\begin{align*}
    \BP(\omega\in B: \BE[X_1\,|\,\fF]-\BE[X_2\,|\,\fF]\geq  \epsilon)=0.
\end{align*}
Thus, we have $\BE[X_1\,|\,\fF]\leq \BE[X_2\,|\,\fF]$ almost surely on $B$.
\end{proof}

\begin{lemma}[{\citet[Theorem 4.1.13]{durrett2019probability}}]\label{lem:ce-exchange}
 If $\fF_0\subset \fF$, then $\BE[X\,|\, \fF_0]=\BE[ \BE[X\,|\,\fF]\,|\,\fF_0]$.   
\end{lemma}

\begin{lemma}[{\citet[Theorem 4.1.10]{durrett2019probability}}]\label{lem:ce-jensen}
If the function $\phi$ is convex and the random variable $X$ satisfies $\BE|X|, \BE|\phi(X)|<\infty$, then
\begin{align*}
    \phi(\BE[X\,|\,\fF])\leq \BE[\phi(X)\,|\,\fF].
\end{align*}
\end{lemma}

Then we provide the proof of Theorem \ref{thm}.

\begin{proof}[Proof of Theorem \ref{thm}]
Recall that
\begin{align*}
    \Omega_\delta= \{\omega\in \Omega_0: \forall  t\leq T, \bigg|\sum_{k=0}^{t-1}\bar{r}_k\inner{\vDelta_k}{\vx_k-\vx_\star}\bigg|< b_t\} \in \fF_0,
\end{align*}
which satisfies $\BP(\Omega_\delta)\geq 1-\delta$ by Lemma \ref{lem:noise-from-g}.
Thus, Proposition \ref{prop} can be expressed as follows: for any $\omega\in \Omega_\delta$, we have
\begin{align*}
    f(\bar{\vx}_t)-f(\vx_\star)\leq   16\cdot \frac{\theta_{t,\delta}(\bar{r}_t+s_0)(\sqrt{G_{t-1}}+ Ld+L\sqrt{dt}\,)}{\sum_{k=0}^{t-1}\bar{r}_k/\bar{r}_t},
\end{align*}
which holds for all $t\leq T$.
Combining Proposition \ref{prop} and equation (\ref{eq:tau}), we obtain that for any $\omega\in \Omega_\delta$, the following inequality holds:
\begin{align*}
  f(\bar{\vx}_{\tau_T}) -f(\vx_\star)  \leq c_0\cdot \frac{\theta_{\tau_T,\delta}(\bar{r}_{\tau_T}+s_0)(\sqrt{G_{\tau_T-1}}+ Ld+L\sqrt{dt}\,)}{T}\log_+\bigg(\frac{\bar{r}_{\tau_T}}{r_\epsilon}\bigg),
\end{align*}
where $c_0$ is a constant and
$\tau_T=\arg\max_{t\leq T}\sum_{k=0}^{t-1}\bar{r}_k/\bar{r}_t$.
Note that $\theta_{t,\delta}$, $\bar{r}_t$ and $G_t$ are non-deceasing w.r.t $t$.
Since $\tau_T\leq T$, the bound becomes
\begin{align*}
    f(\bar{\vx}_{\tau_T}) -f(\vx_\star)  \leq c_0\cdot\frac{\theta_{T,\delta}(\bar{r}_{T}+s_0)(\sqrt{G_{T-1}}+Ld+L\sqrt{dT}\,)}{T}\log_+\bigg(\frac{\bar{r}_{T}}{r_\epsilon}\bigg).
\end{align*}
Note that the diameter is $D_\fX$ and $r_\epsilon\leq D_\fX$. 
Thus, we derive $\bar{r}_T\leq  D_\fX$.
Substituting this into the bound, we obtain
\begin{align*}
    f(\bar{\vx}_{\tau_T}) -f(\vx_\star)  \leq c_0\cdot  \frac{\theta_{T,\delta}D_\fX(\sqrt{G_{T-1}}+ Ld+L\sqrt{dT}\,)}{T}\log_+\bigg(\frac{D_\fX}{r_\epsilon}\bigg).
\end{align*}
Recall that $\fF_\delta = \{A: A\subset \Omega_\delta\} \cap \fF_0$ is a sigma field satisfying $\fF_\delta\subset\fF_0$ and $\Omega_\delta\in \fF_\delta$.
Then, applying Lemma \ref{lem:ce-inequality}, for any $\omega\in \Omega_\delta$, we have
\begin{align}\label{eq:thm1}
       \BE[f(\bar{\vx}_{\tau_T}) -f(\vx_\star)\,|\,\fF_\delta]  \leq c_0\cdot  \frac{\theta_{T,\delta}D_\fX(\BE[\sqrt{G_{T-1}}\,|
       \,\fF_\delta]+ Ld+L\sqrt{dT}\,)}{T}\log_+\bigg(\frac{D_\fX}{r_\epsilon}\bigg).
\end{align}

Using Lemma \ref{lem:upper-bound-for-g} and Lemma \ref{lem:ce-exchange}, the conditional expectation of $G_{T-1}$ can be bounded as follows
\begin{align*}
    \BE[G_{T-1}\,|\,\fF_\delta]= \BE[\BE[G_{T-1}]\,|\,\fF_\delta]=
    \sum_{k=0}^{T-1}\BE[\BE[\norm{\vg_k}^2]\,|\,\fF_\delta] \leq cL^2dT.
\end{align*}
Since the square root function is concave, applying Jensen's inequality in Lemma \ref{lem:ce-jensen} gives
\begin{align*}
    \BE[\sqrt{G_{T-1}}\,|\,\fF_\delta]\leq \sqrt{\BE[G_{T-1}\,|\,\fF_\delta]} \leq L\sqrt{cdT}.
\end{align*}
Substituting this back into equation (\ref{eq:thm1}),  for any $\omega\in \Omega_\delta$, we have
\begin{align}
       \BE[f(\bar{\vx}_{\tau_T}) -f(\vx_\star)\,|\,\fF_\delta]  \leq c_1\bigg(\frac{d}{T}+\frac{\sqrt{d}}{\sqrt{T}}\bigg)\theta_{T,\delta}L D_\fX\log_+\bigg(\frac{D_\fX}{r_\epsilon}\bigg),
\end{align}
where $c_1$ is a constant.
\end{proof}

\section{The Proofs for Section \ref{sec:unbounded}}\label{apx:unbounded}

We provide the detailed proofs for results without the bounded domain assumption.

\subsection{Proof of Proposition \ref{prop:unbounded-1}}\label{apx:unbounded-prop1}
For simplicity, we define the following stopping time
\begin{align*}
    \zeta \triangleq \min\{t\in \BN \mid \bar{r}_t>3s_0\}.
\end{align*}
Using this stopping time, we introduce the modified step size
\begin{align*}
    \Tilde{\eta}_t=\eta_t\cdot \BI(t< \zeta) ,
\end{align*}
where the indicator function $\BI(t< \zeta)$ equals $1$ if $t<\zeta$, and $0$ otherwise.

Before proving the proposition, we first present and prove several supporting lemmas.
\begin{lemma}\label{lem:unbounded-1}
Given $T\in \BN_+$, for any $t\leq T$, the following inequality holds
\begin{align*}
    \sum_{k=0}^{t}\Tilde{\eta}_k^2\norm{\vg_k}^2 \leq \frac{s_0^2}{2}.
\end{align*}
\end{lemma}
\begin{proof}[Proof of Lemma \ref{lem:unbounded-1}]
By the definition of $\Tilde{\eta}_k$ and using $\norm{\vg_k}^2 = G_k-G_{k-1}$, we can bound the term as 
\begin{align}\label{eq:unbounded-1-1}
   \sum_{k=0}^{t}\Tilde{\eta}_k^2\norm{\vg_k}^2  \leq   \sum_{k=0}^{\zeta-1}\eta_k^2\norm{\vg_k}^2 =\sum_{k=0}^{\zeta-1} \frac{\bar{r}_k^2 }{G_k^\prime}\cdot \norm{\vg_k}^2 = \sum_{k=0}^{\zeta-1} \frac{\bar{r}_k^2(G_k-G_{k-1}) }{G_k^\prime}\leq \bar{r}_{\zeta-1}^2 \sum_{k=0}^{\zeta-1} \frac{G_k-G_{k-1} }{G_k^\prime},
\end{align}
where $G_{-1}=0$.
The lower bound for $G_k^\prime$ is given by
\begin{align*}
    G_k^\prime \geq 8^4\theta_{T,\delta}(G_{t-1}+2d^2\bar{L}^2)\log_+^2\bigg(\frac{(k+1)d^2\bar{L}^2+d^2\bar{L}^2}{d^2\bar{L}^2}\bigg)\geq  8^4\theta_{T,\delta}(G_{k}+d^2\bar{L}^2)\log_+^2\bigg(\frac{G_k+d^2\bar{L}^2}{d^2\bar{L}^2}\bigg),
\end{align*}
where the last inequality  follows from $\norm{\vg_k}\leq Ld$ ( See Lemma \ref{lem:upper-bound-for-g}).
Substituting this bound into (\ref{eq:unbounded-1-1}), we obtain
\begin{align}\label{eq:unbounded-1-2}
      \sum_{k=0}^{t}\Tilde{\eta}_k^2\norm{\vg_k}^2 \leq \frac{\bar{r}_{\zeta-1}^2}{8^4\theta_{T,\delta}}\cdot \sum_{k=0}^{\zeta-1} \frac{G_{k}-G_{k-1}}{(G_{k}+d^2\bar{L}^2)\log_+^2\bigg(\frac{G_k+d^2\bar{L}^2}{d^2\bar{L}^2}\bigg)}\leq \frac{\bar{r}_{\zeta-1}^2}{8^4\theta_{T,\delta}}\leq \frac{9s_0^2}{8^4\theta_{T,\delta}}\leq \frac{s_0^2}{2}.
\end{align}
The second inequality holds by applying Lemma \ref{lem:dog-unbounded-1} with $a_k=G_k+d^2\bar{L}^2$.
\end{proof}

\begin{lemma}\label{lem:unbounded-2}
For any $\delta\in (0, 1)$, the following inequality holds
\begin{align*}
    \BP\bigg(\exists t\leq T: \bigg|\sum_{k=0}^{t-1}\Tilde{\eta}_k\inner{\vDelta_k}{\vx_k-\vx_\star}\bigg|>s_0^2\bigg)\leq \delta,
\end{align*}
\end{lemma}
\begin{proof}[Proof of Lemma \ref{lem:unbounded-2}]
We use the filtration 
    $\fF_k=\sigma(\vv_i, \vxi_i, 0\leq i\leq k)$ for $k\in \BN$ and $\fF_{-1}=\{\emptyset,\Omega_0\}$ as defined in Appendix \ref{apx:noise-from-g}.
Note that $\Tilde{\eta}_k \in \fF_{k-1}$.    
Define the stochastic processes $(Z_k,k\in \BN)$ and $(\hat{Z}_k,k\in \BN)$ as follows
\begin{align*}
Z_k=\Tilde{\eta}_k\inner{\vDelta_k}{\vx_k-\vx_\star} \in \fF_k \quad\text{and}\quad \hat{Z}_k = \Tilde{\eta}_k\inner{\nabla f_{\mu_k}(\vx_k)}{\vx_k-\vx_\star}\in \fF_{k-1},
\end{align*}
where $\vDelta_k=\nabla f_{\mu_k}(\vx_k)-\vg_k$.
From this, we observe
\begin{align*}
    \BE[Z_k|\fF_{k-1}] = \bar{\eta}_k\cdot \BE[\inner{\nabla f_{\mu_k}(\vx_k)-\vg_k}{\vx_k-\vx_\star}|\fF_{k-1}]=0,\quad \text{where}\quad k\in\BN.
\end{align*}
Thus,  $(Z_k,\fF_k, k\in \BN)$ is a martingale difference process.

Using the definition of $\Tilde{\eta}_k$ and the fact that $\bar{s}_t\leq \bar{r}_t+s_0$, we bound $|Z_k|$ as follows
\begin{align*}
      |Z_k|\leq \Tilde{\eta}_ks_k\norm{\vDelta_k}\leq  \frac{\bar{r}_{\zeta-1}\bar{s}_{\zeta-1}\norm{\vDelta_k}}{\sqrt{G_k^\prime}}\leq\frac{12s_0^2\norm{\vDelta_k}}{\sqrt{G_k^\prime}}.
\end{align*}
In Appendix \ref{apx:noise-from-g}, we have shown  $\norm{\vDelta_k}\leq 2Ld$.
Moreover, we have $G_k'\geq 16\cdot 8^4 \theta_{T,\delta}^2 d^2\bar{L}^2$.
Thus, we further have
\begin{align*}
    |Z_k| \leq \frac{6s_0^2}{8^2 \theta_{T,\delta}}.
\end{align*}
Similarly, $|Z_k|$ and $|\hat{Z}_k|$ are both bounded above by $6s_0^2/(8^2\theta_{T,\delta})$.
Using Lemma \ref{lem:dog-unbounded-2} with $c=6s_0^2/(8^2\theta_{T,\delta})$, we obtain
\begin{align}\label{eq:unbounded-2-1}
  \BP\bigg(\exists t\leq T: \bigg|\sum_{k=0}^{t-1}Z_k\bigg|>4\sqrt{\theta_{t,\delta}\sum_{k=0}^{t-1}(Z_k-\hat{Z}_k)^2+c^2\theta_{t,\delta}^2}\bigg)\leq \delta.
\end{align}
The upper bound for $\sum_{k=0}^{t-1}(Z_k-\hat{Z}_k)^2$ is given by
\begin{align}\label{eq:unbounded-2-2}
    \sum_{k=0}^{t-1}(Z_k-\hat{Z}_k)^2 =\sum_{k=0}^{t-1}\Tilde{\eta}_k^2(\inner{\vg_k}{\vx_k-\vx_\star})^2\leq \sum_{k=0}^{t-1}\Tilde{\eta}_k^2s_k^2 \norm{\vg_k}^2 \leq \bar{s}_{\zeta-1}^2\sum_{k=0}^{t-1}\Tilde{\eta}_k^2\norm{\vg_k}^2 \leq (4s_0)^2\sum_{k=0}^{t-1}\Tilde{\eta}_k^2\norm{\vg_k}^2 \leq  \frac{12^2s_0^4}{8^4\theta_{T,\delta}},
\end{align}
where  the second inequality follows from $\bar{s}_k\leq \bar{r}_k+s_0$ and the last inequality follows from (\ref{eq:unbounded-1-2}).
Plugging (\ref{eq:unbounded-2-2}) into (\ref{eq:unbounded-2-1}), we get
\begin{align*}
    4\sqrt{\theta_{t,\delta}\sum_{k=0}^{t-1}(Z_k-\hat{Z}_k)^2+c^2\theta_{t,\delta}^2} \leq 4\sqrt{\theta_{t,\delta} \frac{12^2s_0^4}{8^4\theta_{T,\delta}} +\theta_{t,\delta}^2 \frac{6^2s_0^4}{8^4 \theta_{T,\delta}^2} }\leq 4\sqrt{ \frac{12^2s_0^4}{8^4} + \frac{6^2s_0^4}{8^4 }}\leq s_0^2.
\end{align*}
Thus, we have
\begin{align*}
     \BP\bigg(\exists t\leq T: \bigg|\sum_{k=0}^{t-1}\Tilde{\eta}_k\inner{\vDelta_k}{\vx_k-\vx_\star}\bigg|>s_0^2\bigg)\leq \delta.
\end{align*}

\end{proof}

\begin{lemma}\label{lem:unbounded-3}
Given $T\in \BN_+$, for any $t\leq T$, the following inequality holds
\begin{align*}
  \sum_{k=0}^{t}\Tilde{\eta}_k\inner{\nabla f_{\mu_k}(\vx_k)}{\vx_\star-\vx_k}\leq \frac{s_0^2}{4}.
\end{align*}
\end{lemma}
\begin{proof}[Proof of Lemma \ref{lem:unbounded-3}]
By Lemma \ref{lem:error}, we have
\begin{align*}
    \inner{\nabla f_{\mu_k}(\vx_k)}{\vx_\star-\vx_k}\leq f_{\mu_k}(\vx_\star)-f_{\mu_k}(\vx_k)  \leq  f(\vx_\star) - f(\vx_k) +  2L\mu_k\leq 2L\mu_k,
\end{align*}
where the last inequality holds since $f(\vx_\star)=\inf_{\vx\in \fX} f(\vx)$.
Thus, the term can be bounded by
\begin{align*}
     \sum_{k=0}^{t}\Tilde{\eta}_k\inner{\nabla f_{\mu_k}(\vx_k)}{\vx_\star-\vx_k}  =  \sum_{k=0}^{\min(\zeta-1, t)}\eta_k\inner{\nabla f_{\mu_k}(\vx_k)}{\vx_\star-\vx_k}\leq 2L\sum_{k=0}^{\min(\zeta-1, t)} \eta_k\mu_k.
\end{align*}
Since $\mu_k=d\bar{r}_k/(k+1)^2$ and $G_k^\prime \geq 16\cdot 8^4d^2\bar{L}^2$, we have
\begin{align*}
         \sum_{k=0}^{t}\Tilde{\eta}_k\inner{\nabla f_{\mu_k}(\vx_k)}{\vx_\star-\vx_k} \leq  \frac{2L}{4\cdot 8^2 \bar{L}}
        \sum_{k=0}^{\min(\zeta-1, t)}\frac{\bar{r}_k^2}{(k+1)^2} \leq \frac{ \bar{r}_{\zeta-1}^2}{2\cdot 8^2 }\sum_{k=0}^{\min(\zeta-1,t)}\frac{1}{(k+1)^2} \leq \frac{3\pi^2 s_0^2}{4\cdot 8^2}\leq \frac{s_0^2}{4},
\end{align*}
where the third inequality holds by the fact that $\sum_{k=1}^\infty 1/k^2=\pi^2/6$ and $\bar{r}_{\zeta-1}\leq 3s_0$.
\end{proof}

Based on Lemmas \ref{lem:unbounded-1}, \ref{lem:unbounded-2}, and \ref{lem:unbounded-3}, we now establish Proposition \ref{prop:unbounded-1}.

\begin{proof}[Proof of Proposition \ref{prop:unbounded-1}]
Given any $\delta>0$, we define the following set
\begin{align*}
    \Tilde{\Omega}_\delta \triangleq \bigg\{\omega\in \Omega_0: \forall t\leq T,  \bigg|\sum_{k=0}^{t-1}\Tilde{\eta}_k\inner{\vDelta_k}{\vx_k-\vx_\star}\bigg|\leq s_0^2\bigg\}.
\end{align*}
By Lemma \ref{lem:unbounded-2}, we have $\BP(\Tilde{\Omega}_\delta)\geq 1-\delta$.

We use induction to complete the proof.
For the initial step, we have $\bar{r}_0=r_\epsilon\leq 3s_0$.
For the induction step, we assume $\bar{r}_{t-1}\leq 3s_0$, which is equivalent to $\zeta>t-1$.
From (\ref{eq:unbounded-s}), we have    
 \begin{align*}
    s_t^2-s_0^2 &\leq \sum_{k=0}^{t-1}\eta_k^2\norm{\vg_k}^2+2\sum_{k=0}^{t-1}\eta_k\inner{\vDelta_k}{\vx_k-\vx_\star} +2\sum_{k=0}^{t-1}\eta_k\inner{\nabla f_{\mu_k}(\vx_k)}{\vx_\star-\vx_k}\\
    &= \sum_{k=0}^{t-1}\Tilde{\eta}_k^2\norm{\vg_k}^2+2\sum_{k=0}^{t-1}\Tilde{\eta}_k\inner{\vDelta_k}{\vx_k-\vx_\star} +2\sum_{k=0}^{t-1}\Tilde{\eta}_k\inner{\nabla f_{\mu_k}(\vx_k)}{\vx_\star-\vx_k},
 \end{align*}
 where the equality holds since $\zeta>t-1$.
Applying Lemmas \ref{lem:unbounded-1}, \ref{lem:unbounded-2}, and \ref{lem:unbounded-3}, for any $\omega\in \Tilde{\Omega}_\delta$, we derive
 \begin{align*}
     s_t^2-s_0^2\leq \frac{1}{2}s_0^2 + 2s_0^2+2\cdot \frac{1}{4}s_0^2=3s_0^2,
 \end{align*}
 which implies that $s_t\leq 2 s_0$ , and consequently $r_t\leq s_t+s_0=3 s_0$.
 By the induction step, we have $\bar{r}_t = \max(\bar{r}_{t-1}, r_t)\leq 3s_0$.
 Therefore, we conclude that for any $\omega\in \Tilde{\Omega}_\delta$, we have $\bar{r}_t\leq 3s_0$ for all $t\leq T$. 
 Equivalently, $ \BP(\bar{r}_T>3 s_0)\leq \delta$.
\end{proof}

\subsection{Proof of Proposition \ref{prop:unbounded-2}}\label{apx:unbounded-prop2}
\begin{proof}[Proof of Proposition  \ref{prop:unbounded-2}]
Since $G_t^\prime \geq G_t$, using the results in \citet[Lemma 3.4]{ivgi2023dog}, we obtain Lemma \ref{lem:weighted-regret} holds by replacing $G_t$ with $G_t^\prime$.
Besides, Lemma \ref{lem:noise-from-g} still holds.
For Lemma \ref{lem:noise-from-mu}, since $\mu_t=d\bar{r}_t/(t+1)^2$ in (\ref{eq:mu-unbounded}), the noise from $\mu$ can be re-bounded as 
\begin{align*}
    \sum_{k=0}^{t-1} 2L \bar{r}_k\mu_k = \sum_{k=0}^{t-1} \frac{2Ld\bar{r}_k^2}{(k+1)^2}\leq 2Ld\bar{r}_{t-1}^2 
    \sum_{k=0}^{t-1}\frac{1}{(k+1)^2}\leq 4Ld\bar{r}_{t-1}^2,
\end{align*}
where we use the fact that $\sum_{k=1}^\infty 1/k^2=\pi^2/6$.

Combining these modified Lemmas along with equations (\ref{eq:gapsmooth}) and (\ref{eq:threecomponentssmooth}), we obtain that, with probability at least $1-\delta$,  the upper bound for  $f(\bar{\vx}_t)-f(\vx_\star)$ is 
\begin{align*}
\frac{(2\bar{s}_t+\bar{r}_t)\sqrt{G_{t-1}^\prime} + 8\bar{s}_t\sqrt{\theta_{t,\delta}G_{t-1}+4L^2d^2\theta_{t,\delta}^2}\,+4Ld\bar{r}_t}{\sum_{k=0}^{t-1}\bar{r}_k/\bar{r}_t},
\end{align*} 
where we ignore the constants and use the fact that $\bar{r}_{t-1}\leq \bar{r}_t$ and $\bar{s}_{t-1}\leq \bar{s}_t$.
Applying the inequality $\sqrt{a^2+b^2}\,\leq a+b$, the gap simplifies to
\begin{align*}
         \frac{(2\bar{s}_t+\bar{r}_t)\sqrt{G_{t-1}^\prime} + 8\bar{s}_t(\theta_{t,\delta}\sqrt{G_{t-1}}+2\theta_{t,\delta}Ld)+4Ld\bar{r}_t}{\sum_{k=0}^{t-1}\bar{r}_k/\bar{r}_t}.
\end{align*}
Finally, using the fact that $G_t^\prime \geq G_t$ and the
triangle inequality $\bar{s}_t\leq \bar{r}_t+s_0$, the gap becomes
\begin{align*}
    20\cdot \frac{\theta_{t,\delta}(\bar{r}_t+s_0)(\sqrt{G_{t-1}^\prime}+ Ld)}{\sum_{k=0}^{t-1}\bar{r}_k/\bar{r}_t}.
\end{align*}

\end{proof}

\subsection{Proof of Theorem \ref{thm2}}\label{apx:thm2}
We first establish the existence and uniqueness of the conditional expectation $\BE[f(\bar{\vx}_{\tau_T})-f(\vx_\star)\,|\, \Tilde{\fF}_\delta]$.
According to \citet[Chapter 4.1]{durrett2019probability}, it suffices to verify that $\Tilde{\fF}_\delta\subset \Tilde{\fF}_0$ and $\BE|f(\bar{\vx}_{\tau_T})-f(\vx_\star)|<\infty$.
By definition, we recall that $\Tilde{F}_\delta = \{A:A\subset \Tilde{\Omega}_\delta\,\cap\, \hat\Omega_\delta\}\cap \Tilde{\fF}_0$.
Thus, we have $\Tilde{\fF}_\delta\subset \Tilde{\fF}_0$.
Next, from inequality (\ref{eq:unbounded-s}), we have
\begin{align*}
     s_t^2-s_0^2 \leq & \sum_{k=0}^{t-1}\eta_k^2\norm{\vg_k}^2+2\sum_{k=0}^{t-1}\eta_k\inner{\vDelta_k}{\vx_k-\vx_\star}  +2\sum_{k=0}^{t-1}\eta_k\inner{\nabla f_{\mu_k}(\vx_k)}{\vx_\star-\vx_k},
\end{align*}
where $t=1,2,\dots, T$.
Using the upper bounds $\norm{\vg_k}\leq Ld$ (Lemma \ref{lem:upper-bound-for-g}), $\norm{\nabla f_{\mu_k}(\vx_k)}\leq Ld$ and  $\norm{\vDelta_k}\leq 2Ld$ (Appendix~\ref{apx:noise-from-g}), we obtain 
\begin{align*}
     s_t^2-s_0^2 &\leq \sum_{k=0}^{t-1}\eta_k^2\norm{\vg_k}^2+2\sum_{k=0}^{t-1}\eta_k\norm{\vDelta_k}\norm{\vx_k-\vx_\star}  +2\sum_{k=0}^{t-1}\eta_k\norm{\nabla f_{\mu_k}(\vx_k)}\norm{\vx_\star-\vx_k}\\
    &\leq  L^2d^2\sum_{k=0}^{t-1}\eta_k^2 + 6Ld\bar{s}_{t-1}\sum_{k=0}^{t-1}\eta_k.
\end{align*}
Since $\eta_k=\bar{r}_k/\sqrt{G_k^\prime}$ with $G_k^\prime\geq d^2L^2$, it follows that
\begin{align*}
     s_t^2-s_0^2 \leq T (\bar{r}_{t-1}^2 + 6\bar{r}_{t-1}\bar{s}_{t-1}).
\end{align*}
Given that $\bar{r}_0=r_\epsilon\leq 3s_0<\infty$, using mathematical  induction, we can obtain that $\bar{s}_t<\infty$ and $\bar{r}_t<\infty$ for all $t\leq T$.
Finally, by the Lipschitz continuity of $f(\cdot)$, we obtain
\begin{align*}
    |f(\bar{\vx}_{\tau_T})-f(\vx_\star)|\leq L\norm{\bar{\vx}_{\tau_T}-\vx_\star}\leq L\bar{s}_T<\infty.
\end{align*}
Thus, we obtain $\BE|f(\bar{\vx}_{\tau_T})-f(\vx_\star)|<\infty$.
Therefore, we conclude that the conditional expectation $\BE[f(\bar{\vx}_{\tau_T})-f(\vx_\star)\,|\, \Tilde{\fF}_\delta]$ exists and is unique.

Now, we provide the proof for Theorem \ref{thm2}.
\begin{proof}[Proof of Theorem \ref{thm2}]
Recall that
\begin{align*}
\Tilde{\Omega}_\delta &= \bigg\{\omega\in \Omega_0: \forall t\leq T,  \bigg|\sum_{k=0}^{t-1}\Tilde{\eta}_k\inner{\vDelta_k}{\vx_k-\vx_\star}\bigg|\leq s_0^2\bigg\},\\
    \hat{\Omega}_\delta &= \{\omega\in \Omega_0: \forall  t\leq T, \bigg|\sum_{k=0}^{t-1}\bar{r}_k\inner{\vDelta_k}{\vx_k-\vx_\star}\bigg|< b_t\} ,
\end{align*}
which satisfies $\Tilde{\BP}(\Tilde{\Omega}_\delta)\geq 1-\delta$ and $\Tilde{\BP}(\hat{\Omega}_\delta)\geq 1-\delta$.
From  Proposition \ref{prop:unbounded-2},  for any $\omega\in \Tilde{\Omega}_\delta \cap \hat{\Omega}_\delta$, we have
\begin{align*}
    f(\bar{\vx}_t)-f(\vx_\star)\leq   20\cdot \frac{\theta_{t,\delta}(\bar{r}_t+s_0)(\sqrt{G_{t-1}^\prime}+ Ld)}{\sum_{k=0}^{t-1}\bar{r}_k/\bar{r}_t},
\end{align*}
which holds for all $t\leq T$.
Combining Proposition \ref{prop} and equation (\ref{eq:tau}), we obtain that for any $\omega\in \Tilde{\Omega}_\delta \cap \hat{\Omega}_\delta$, the following inequality holds:
\begin{align*}
  f(\bar{\vx}_{\tau_T}) -f(\vx_\star)  \leq c_0\cdot \frac{\theta_{\tau_T,\delta}(\bar{r}_{\tau_T}+s_0)(\sqrt{G_{\tau_T-1}^\prime}+ Ld)}{T}\log_+\bigg(\frac{\bar{r}_{\tau_T}}{r_\epsilon}\bigg),
\end{align*}
where $c_0$ is a constant and
$\tau_T=\arg\max_{t\leq T}\sum_{k=0}^{t-1}\bar{r}_k/\bar{r}_t$.
Note that $\theta_{t,\delta}$, $\bar{r}_t$ and $G_t$ are non-deceasing w.r.t $t$.
Since $\tau_T\leq T$, the bound becomes
\begin{align*}
    f(\bar{\vx}_{\tau_T}) -f(\vx_\star)  \leq c_0\cdot\frac{\theta_{T,\delta}(\bar{r}_{T}+s_0)(\sqrt{G_{T-1}^\prime}+Ld)}{T}\log_+\bigg(\frac{\bar{r}_{T}}{r_\epsilon}\bigg).
\end{align*}
From Proposition \ref{prop:unbounded-1}, we have  $\bar{r}_T\leq  3s_0$ for any $\omega\in \Tilde{\Omega}_\delta \cap \hat{\Omega}_\delta$.
Substituting this into the bound, we obtain
\begin{align*}
    f(\bar{\vx}_{\tau_T}) -f(\vx_\star)  \leq 4c_0\cdot  \frac{\theta_{T,\delta}s_0(\sqrt{G_{T-1}^\prime}+ Ld)}{T}\log_+\bigg(\frac{3s_0}{r_\epsilon}\bigg).
\end{align*}
Recall that $\Tilde{\fF}_\delta = \{A: A\subset  \Tilde{\Omega}_\delta \cap \hat{\Omega}_\delta\} \cap \Tilde{\fF}_0$ is a sigma field satisfying $\Tilde{\fF}_\delta\subset\Tilde{\fF}_0$ and $\Tilde{\Omega}_\delta\cap \hat{\Omega}_\delta \in \Tilde{\fF}_\delta$.
Then, applying Lemma \ref{lem:ce-inequality}, for any $\omega\in \Tilde{\Omega}_\delta \cap \hat{\Omega}_\delta$, we have
\begin{align}\label{eq:thm1-1}
       \BE[f(\bar{\vx}_{\tau_T}) -f(\vx_\star)\,|\,\Tilde{\fF}_\delta]  \leq 4c_0\cdot  \frac{\theta_{T,\delta}s_0(\BE[\sqrt{G_{T-1}^\prime}\,|\, \Tilde{\fF}_\delta]+ Ld)}{T}\log_+\bigg(\frac{3s_0}{r_\epsilon}\bigg).
\end{align}

Using Lemma \ref{lem:upper-bound-for-g} and Lemma \ref{lem:ce-exchange}, the conditional expectation of $G_{T-1}$ can be bounded as follows
\begin{align*}
    \BE[G_{T-1}^\prime\,|\,\Tilde{\fF}_\delta]= \BE[\BE[G_{T-1}^\prime]\,|\,\Tilde{\fF}_\delta]=&
    8^4\theta_{T,\delta}\log_+^2(T+1)\bigg(\sum_{k=0}^{T-1}\BE[\BE[\norm{\vg_k}^2]\,|\,\Tilde{\fF}_\delta]+16\theta_{ T,\delta}d^2\bar{L}^2\bigg) \\
    \leq  &8^4\theta_{T,\delta}^2\log_+^2(T+1)(cL^2dT+16d^2\bar{L}^2).
\end{align*}
Since the square root function is concave, applying Jensen's inequality in Lemma \ref{lem:ce-jensen} gives
\begin{align*}
    \BE[\sqrt{G_{T-1}^\prime}\,|\,\Tilde{\fF}_\delta]\leq \sqrt{\BE[G_{T-1}^\prime\,|\,\Tilde{\fF}_\delta]} \leq 8^2\theta_{T,\delta}\log_+(T+1)\sqrt{cL^2dT+16d^2\bar{L}^2}\leq 8^2\theta_{T,\delta}\log_+(T+1)(\sqrt{c}L\sqrt{dT}+4d\bar{L}).
\end{align*}
Substituting this back into equation (\ref{eq:thm1-1}),  for any $\omega\in \Tilde{\Omega}_\delta \cap \hat{\Omega}_\delta$, we have
\begin{align}
       \BE[f(\bar{\vx}_{\tau_T}) -f(\vx_\star)\,|\,\Tilde{\fF}_\delta]  \leq c_1\bigg(\frac{d}{T}\cdot(L+\bar{L})+\frac{\sqrt{d}}{\sqrt{T}}\cdot L\bigg)\alpha_{T,\delta} s_0\log_+\bigg(\frac{s_0}{r_\epsilon}\bigg),
\end{align}
where $c_1$ is a constant and $\alpha_{T,\delta}=\theta_{T,\delta}\log_+(T+1)$.
Moreover, we have
\begin{align*}
    \Tilde{\BP}(\Tilde{\Omega}_\delta \cap \hat{\Omega}_\delta)\geq \Tilde{\BP}(\Tilde{\Omega}_\delta)+\Tilde{\BP}( \hat{\Omega}_\delta)-1\geq 1-2\delta.
\end{align*}
\end{proof}

\subsection{Proof of Theorem \ref{thm3}}
\begin{proof}\label{apx:thm3} 
For given $T\geq 2$, we let $\vxi\sim \Xi$, where $\Xi$ is a Bernoulli distribution such that
\begin{align*}
    \BP(\vxi=0)= 1-\frac{1}{T} \quad \text{and}\quad \BP(\vxi=1) =\frac{1}{T}.
\end{align*}

We first define the function $f_1:\BR^d\to\BR$ as
\begin{align*}
    f_1(\vx) = L\norm{\vx}_1,
\end{align*}
where $\|\cdot\|$ is the $\ell_1$-norm.
We introduce the SZO for function $f_1$ such that for given $\vx\in \BR^d$ and $\vy\in \BR^d$, it returns the evaluation $F_1(\vx;\vxi)$ and $F_1(\vy;\vxi)$ satisfying $F_1(\vx;\vxi)=L\norm{\vx}_1$ and $F_2(\vy;\vxi)=L\norm{\vy}_1$ for all $\vxi$ drawn from $\Xi$.
We can verify such oracle satisfies Assumption \ref{asm:oracle}.
Moreover, the function $F_1(\cdot\,;\vxi)$ is convex and Lipschitz continuous with constant $L_1=L$ for all $\vxi$. 

We then define another function $f_2:\BR^d\to\BR$ as
\begin{align*}
    f_2(\vx) = L\norm{\vx-\vu}_1,
\end{align*}
where  $\vu=(1-1/T)\mathbf{1}_d\in \BR^d$.
We introduce the SZO for function $f_2$ such that for given $\vx\in \BR^d$ and $\vy\in \BR^d$, it returns the evaluation $F_2(\vx;\vxi)$ and $F_2(\vy;\vxi)$ satisfying
\begin{align*}
    F_2(\vx;0)=L\norm{\vx}_1, \qquad F_2(\vx;1)=TL\norm{\vx-\vu}_1-(T-1)L\norm{\vx}_1, \\
    F_2(\vy;0)=L\norm{\vy}_1, \qquad F_2(\vx;1)=TL\norm{\vy-\vu}_1-(T-1)L\norm{\vy}_1. 
\end{align*}
Thus, we have $\BE_{\vxi\sim\Xi}[F_2(\vz;\vxi)]=f_2(\vz)$ for all $\vz\in\BR^d$, which satisfies Assumption \ref{asm:oracle}.
Moreover, we can verify that both $ F_2(\vz;0)$ and $ F_2(\vz;1)$ are convex and Lipschitz continuous on $\BR$, where the Lipschitz constant is $L_2=2TL$.

We set the initial point as $\vx_0=\mathbf{1}_d$,
then the probability of obtaining the identical information from two oracles $F_1$ and $F_2$ with $T$ rounds of oracle calls is
\begin{align*}
p= \left(1-\frac{1}{T}\right)^T  \geq  \frac{1}{\mathrm{e}}
\end{align*}
for all $T\geq 2$.
This implies that any SZO algorithm cannot distinguish $f_1(\cdot)$ and $f_2(\cdot)$ in $T$ rounds of SZO calls with probability at least $1/{\rm e}$.
Therefore, a near-optimal SZO algorithm $\fA$ must achieve the nearly tight function value gap for $T$ with probability $1/{\rm e}$.
In other words, algorithm $\fA$ must output $\hat{\vx}$ satisfying
\begin{align*}
f_1(\hat{\vx}) - f_1^\star \leq \theta_1\bigg(\frac{\bar{L}}{\underline{L}}, \frac{\bar{s}}{\underline{s}}, T, d\bigg) \cdot \frac{\sqrt{d}L_1\norm{\vx_0-\vx_{1,*}}_2 }{\sqrt{T}}
\quad\text{and}\quad
f_2(\hat{\vx}) - f_2^\star \leq \theta_2\bigg(\frac{\bar{L}}{\underline{L}}, \frac{\bar{s}}{\underline{s}}, T, d\bigg) \cdot \frac{\sqrt{d}L_2\norm{\vx_0-\vx_{2,*}}_2}{\sqrt{T}},
\end{align*}
where $\vx_{1,\star}\triangleq\argmin_{\vx\in\BR^d} f_1(\vx)=\vzero$, $\vx_{2,\star}\triangleq\argmin_{\vx\in\BR^d} f_2(\vx)=\vu$, and
$\theta_1,\theta_2:\BR^4\to\BR$ are two polylogarithmic functions.
Since we have shown $L_1=L$ and $L_2=2TL$,
the above upper bounds on function value gap can be written as
\begin{align*}
\norm{\hat{\vx}}_1\leq \theta_1\bigg(\frac{\bar{L}}{\underline{L}}, \frac{\bar{s}}{\underline{s}}, T, d\bigg)\cdot \frac{d}{\sqrt{T}}
\qquad \text{and} \qquad
\norm{\hat{\vx}-\vu}_1 \leq \theta_2\bigg(\frac{\bar{L}}{\underline{L}}, \frac{\bar{s}}{\underline{s}}, T, d\bigg)\cdot\sqrt{d}\norm{\mathbf{1}_d-\vu}_2\sqrt{T}.
\end{align*}
Recall than $\vu=(1-1/T)\mathbf{1}_d$, then the point $\hat{\vx}$ must satisfy
\begin{align*}
 \norm{\hat{\vx}}_1\leq \theta_1\bigg(\frac{\bar{L}}{\underline{L}}, \frac{\bar{s}}{\underline{s}}, T, d\bigg)\cdot \frac{d}{\sqrt{T}}\quad\text{and}\quad \norm{\hat{\vx}}_1\geq d-\frac{d}{T}-\theta_2\bigg(\frac{\bar{L}}{\underline{L}}, \frac{\bar{s}}{\underline{s}}, T, d\bigg)\cdot \frac{d}{\sqrt{T}}.
\end{align*}
Since the functions $\theta_1$ and $\theta_2$ are polylogarithmic, there exist a sufficient large $T$ such that 
\begin{align*}
\theta_1\bigg(\frac{\bar{L}}{\underline{L}}, \frac{\bar{s}}{\underline{s}}, T, d\bigg)\cdot \frac{d}{\sqrt{T}} <  d-\frac{d}{T}-\theta_2\bigg(\frac{\bar{L}}{\underline{L}}, \frac{\bar{s}}{\underline{s}}, T, d\bigg)\cdot \frac{d}{\sqrt{T}},
\end{align*}
which leads to contradiction.    
Hence, we conclude that achieving an ideal parameter-free stochastic zeroth-order algorithm described in the theorem is impossible.
\end{proof}

\newpage

\end{document}